\documentclass[a4paper,12pt]{amsart}
\usepackage{amsmath, amssymb, amsthm, xcolor, enumerate}
\usepackage[utf8]{inputenc}
\usepackage{csquotes}
\usepackage{tikz-cd}
\usepackage[all]{xy}
\usepackage{cancel}
\usepackage[english]{babel}
\usepackage{graphicx}
\usepackage{adjustbox}
\usepackage[hypertexnames=false]{hyperref}

\newcounter{theorem}
\newtheorem{theorem}[theorem]{Theorem}
\newtheorem{lemma}[theorem]{Lemma}
\newtheorem{prop}[theorem]{Proposition}
\newtheorem{cor}[theorem]{Corollary}

\theoremstyle{definition}
\newtheorem{defn}[theorem]{Definition}

\newtheorem{problem}[theorem]{Problem}

\theoremstyle{remark}
\newtheorem*{remark*}{Remark}
\newtheorem{rmk}[theorem]{Remark}

\numberwithin{equation}{section}

\allowdisplaybreaks

\newcommand{\Z}{\mathcal Z}
\newcommand{\C}{\mathrm{C}^*}

\newcommand{\N}{\mathbb{N}}
\newcommand{\M}{\mathbb{M}}
\newcommand{\dimnuc}{\dim_{\mathrm{nuc}}}

\newcommand{\id}{\mathrm{id}}

\newcommand{\alg}{\mathrm{alg}}
\newcommand{\Aff}{\mathrm{Aff}}
\newcommand{\im}{\operatorname{im}}
\newcommand{\au}{\operatorname{au}}
\newcommand{\Tor}{\operatorname{Tor}}

\newcommand{\ev}{\operatorname{ev}}
\newcommand{\KK}{\operatorname{KK}}
\newcommand{\Th}{\operatorname{Th}}
\newcommand{\Hom}{\operatorname{Hom}}

\title[]{On topologically zero-dimensional morphisms}
\author[J.\ Castillejos]{Jorge Castillejos}
\address{\hskip-\parindent Jorge Castillejos, Instituto de Matem\'aticas, Unidad Cuernavaca, Universidad Nacional Aut\'onoma de M\'exico, 
	 Cuernavaca 62210, Morelos, México.}
\email{jorge.castillejos@im.unam.mx}
\author[R.\ Neagu]{Robert Neagu}
\address{\hskip-\parindent Robert Neagu, Mathematical Institute, University of Oxford, Oxford, OX2 6GG, UK.}
\email{robert.neagu@maths.ox.ac.uk}

\thanks{The first named author was supported by UNAM--PAPIIT IA103124. The second named author was supported by the EPSRC grant EP/R513295/1.}

\begin{document}
	
	\maketitle
	
	\begin{abstract}
		We investigate $^*$-homomorphisms with nuclear dimension equal to zero. 
		In the framework of classification of $^*$-homo-morphisms, we characterise such maps as those that can be approximately factorised through an AF-algebra.
		
		Along the way, we obtain various obstructions for the total invariant of zero-dimensional morphisms and show that in the presence of real rank zero, nuclear dimension zero can be completely determined at the level of the total invariant. We end by characterising when unital embeddings of $\mathcal{Z}$ have nuclear dimension equal to zero.
	\end{abstract}

\numberwithin{theorem}{section}	
	
 \section*{Introduction}
\renewcommand*{\thetheorem}{\Alph{theorem}}
	
Nuclear dimension, introduced by Winter and Zacharias in \cite{nucdim}, is a non-commutative generalisation of covering dimension for topological spaces to $\C$-algebras. 
It is now known that, for simple separable nuclear $\C$-algebras, finite nuclear dimension is equivalent to tensorial absorption of the Jiang-Su algebra $\mathcal{Z}$ (\cite{Win, T, CPOU, CE}). In fact,
the nuclear dimension of a simple $\C$-algebra can only be zero, one or infinite (\cite{CPOU,CE}).
This revolutionary notion played a major role in the classification programme of simple separable unital nuclear $\C$-algebras. 

In recent years, it has become an important task to understand not only regularity properties of $\C$-algebras but also of maps between them. 
This is in part due to the fact that the
proof of many classification results can be divided, roughly speaking, in two important steps usually called ``existence" and ``uniqueness". 
Such steps focus on classifying maps rather than $\C$-algebras;
but by combining them, one can obtain classification results for the algebras themselves.

Nuclear dimension of $^*$-homomorphisms is a natural regularity condition to investigate.
Since nuclear dimension of a $\C$-algebra is defined by means of an approximation property for the identity map, it is immediate how to generalise this notion to $^*$-homomorphisms;  
a line of research that was initiated in \cite{decomprankZstable}. More recently,
Bosa, Gabe, Sims, and White showed that if $A$ is separable and exact, then any nuclear $\mathcal{O}_\infty$-stable $^*$-homomorphism $\theta:A\to B$ has nuclear dimension at most one ({\cite[Theorem C]{nucdimOinftystable}}).
On the other hand, from the main result of \cite{CPOU}, it follows that
if $\theta:A\to B$ is a 
$^*$-homomorphism 
where the domain or codomain is a simple nuclear $\Z$-stable $\C$-algebra,
then the nuclear dimension of $\theta$ is at most one.
Based on this, it is natural to investigate criteria that will allow us to determine when the nuclear dimension of $^*$-homomorphisms between classifiable $\C$-algebras is zero or one.
In \cite{nucdimOinftystable}, the authors pose the following problem.

\begin{problem}[{\cite[Problem 7.3]{nucdimOinftystable}}]\label{theproblem}
Characterise $^*$-homomorphisms with nuclear dimension equal to zero.
\end{problem}

A very recent work shows that under some regularity conditions, unital $^*$-homomorphisms are classified by a new invariant called the \emph{total invariant} \cite{classif}. 
Our objective in this paper will be to analyse Problem \ref{theproblem} for $^*$-homomorphisms covered by the classification theorem in \cite{classif}. We will describe the invariant of classifiable $^*$-homomorphisms with nuclear dimension equal to zero. Then, under some additional hypotheses, we can completely characterise the property of having nuclear dimension equal to zero at the level of the invariant.

Early in the theory, Winter established that a $\C$-algebra has nuclear dimension equal to zero if and only if it is \emph{approximately finite dimensional} (AF) {\cite[Theorem 4.2.3]{Win03}}.
For $^*$-homomorphisms, it is readily seen that a $^*$-homomorphism factoring through an approximately finite dimensional $\C$-algebra has nuclear dimension equal to zero. \emph{Does this condition characterise all zero-dimensional morphisms up to some notion of equivalence?}

The main result of this paper states that this previous condition characterises classifiable zero-dimensional $^*$-homomorphisms $A \to B$ when viewed as maps into the sequence algebra $B_\infty$.

\begin{theorem}\label{factoringintoBinfty}
Let $A$ and $B$ be simple separable unital and nuclear $\C$-algebras such that $A$ is stably finite and satisfies the UCT and  $B$ is $\Z$-stable.
If $\theta:A\to B$ is a unital $^*$-homomorphism and $\iota_B:B\to B_\infty$ is the canonical diagonal inclusion, then the following are equivalent:
\begin{enumerate}
    \item There exist a simple separable unital AF-algebra $C$ and unital $^*$-homomorphisms $\psi: A \to C$ and $\varphi: C \to B_\infty$ such that
    $\iota_B \circ \theta = \varphi \circ \psi$;
    
    \item $\theta$ has nuclear dimension equal to zero.\label{factoring32}
\end{enumerate}
\end{theorem}

In the absence of torsion in $K$-theory, a strengthening of this result can be obtained by assuming \emph{real rank zero} on the domain or codomain of the morphisms.
In either of these situations, 
we can show that the zero-dimensional $^*$-homomorphisms are precisely those which are approximately unitarily equivalent to those factoring through a simple separable unital AF-algebra; and furthermore, this can be identified at the level of the invariant.

\begin{theorem}\label{rr0codomain}
Let $A$ and $B$ be simple separable unital and nuclear $\C$-algebras such that $A$ is stably finite and satisfies the UCT and $B$ is $\Z$-stable of real rank zero with $T(B) \neq \emptyset$ and $K_0(B)$ is torsion free.
Let $\theta:A\to B$ be a unital $^*$-homomorphism. Then the following are equivalent:
\begin{enumerate}

	\item There exist a simple separable unital AF-algebra $C$ and unital $^*$-homomorphisms $\psi: A \to C$ and $\varphi: C \to B$ such that $\theta$ is approximately unitarily equivalent to $\varphi \circ \psi$;\label{factoring1}

    \item $\theta$ has nuclear dimension equal to zero; \label{factoring2}

    \item $K_1(\theta)=0$ and $K_1(\theta;\mathbb{Z}/n)=0$ for all $n\in\mathbb{N}$.\label{factoring3}
\end{enumerate}
\end{theorem}

A somewhat similar result is Theorem \ref{rr0domain} where the condition of real rank zero is imposed on the domain of the morphism.

\medskip
\subsection*{\sc Discussion of methods}\hfill \\

One natural starting point is trying to mimic Winter's proof in \cite{Win03} in the unital case. Loosely speaking, it goes as follows: using an approximating system witnessing nuclear dimension equal to zero, we would like to produce an increasing family $(F_n)_{n\in\mathbb{N}}$ of finite dimensional $\C$-subalgebras of $A$ with dense union.
Let us suppose we have already found $F_n$ and want to obtain a larger finite dimensional $\C$-algebra $F_{n+1}$ approximately containing a finite subset $\mathcal{F}$ of  $A$ up to a given $\epsilon>0$. Nuclear dimension equal to zero and a perturbation argument will give a finite dimensional $\C$-algebra $F$ and a unital $^*$-homomorphism $\phi:F\to A$ such that $\phi(F)$ approximates $\mathcal{F}$ up to $\epsilon$. To make sure $F_{n+1}$ contains $F_n$, we can assume that the matrix units of $F_n$ are contained in the finite set $\mathcal{F}$. Following the proof of {\cite[Theorem 4.2.3]{Win03}}, we can obtain $F_{n+1}$ starting from $\phi(F)$. 

However, when dealing with a general  $^*$-homomorphism $\theta:A\to B$ with nuclear dimension equal to zero, a ``double approximation issue" arises if we try to adopt the same strategy. 
Precisely, the image of the upward approximating maps $\phi(F)$ will possibly lie outside the image of $\theta$ in the bigger $\C$-algebra $B$. Therefore, the strategy employed by Winter will produce a finite dimensional $\C$-algebra $F_{n+1}$ which generally will not be contained in the image of $\theta$. Then, we cannot ensure in the next step that we can include the necessary matrix units in the finite set to be approximated. 
Summarising, we can obtain that for any finite subset $\mathcal{F}$ of $A$ and any $\epsilon>0$, there exists a finite dimensional $\C$-algebra $F$ in $B$ approximating $\theta(\mathcal{F})$ up to $\epsilon$, but we cannot a priori produce any sort of increasing sequence of finite dimensional $\C$-algebras; and hence, no AF-algebra within $B$. 

Instead, we propose a different approach to this problem.\ By a careful analysis at the level of K-theory of the approximations witnessing nuclear dimension equal to zero, we will produce a simple separable unital AF-algebra. Then, after analysing zero-dimensional $^*$-homomorphisms at the level of the total invariant, we will make use of {\cite[Theorem 9.1]{classif}} and {\cite[Theorem 9.3]{classif}} to produce a factorisation of such morphisms via the AF-algebra constructed in the previous step.

The paper is organised as follows.
Section \ref{preliminaries} reviews the definition of nuclear dimension together with some standard simplifications in the case when $\theta:A\to B$ is a unital $^*$-homomorphism with nuclear dimension equal to zero. The last part of Section \ref{preliminaries} collects some facts about the total invariant which are going to be used in the proof of the main results via the classification theorems {\cite[Theorem 9.1]{classif}} and {\cite[Theorem 9.3]{classif}}. Section \ref{ultraprod} then focuses on properties of sequence algebras of finite dimensional $\C$-algebras. The main goal is Proposition \ref{dimgroup}, which shows that if the dimension of the matrix blocks grows fast enough, we can obtain a non-elementary dimension group that can be mapped into the $K_0$-group of the sequence algebra. This is going to provide the key ingredient in the proof of Theorem \ref{factoringintoBinfty}. We then change gears in Section \ref{zerodim} and collect a number of properties of morphisms with nuclear dimension equal to zero at the level of their invariant. Combining the results in Section \ref{ultraprod} and Section \ref{zerodim}, we prove our main results in Section \ref{mainresults}. We finish this paper in Section \ref{embZ} by providing characterisations of unital embeddings of $\Z$ with nuclear dimension equal to zero.

\subsection*{Acknowledgements} 
We thank Stuart White for suggesting this problem and for helpful comments that improved the paper. We also thank the authors of \cite{classif} for kindly sharing a preliminary version of their paper. 
We thank the anonymous referee and Mikkel Munkholm for their suggestions on an earlier version of this paper. 
This article is part of the PhD thesis of the second named author.
	
\section{Preliminaries}\label{preliminaries}

\numberwithin{theorem}{section}

\subsection{Notation}
In the sequel, 
$\M_n$ will denote the $\C$-algebra of $n\times n$-matrices with coefficients in $\mathbb{C}$. In general,
$A$ and $B$ will stand for $\C$-algebras, often assumed to be separable. Then, we will write $T(A)$ for the set of tracial states of $A$ and $QT(A)$ for the set of lower semicontinuous $2$-quasitraces.
The space of all affine functions on $T(A)$ will be denoted by $\Aff(T(A))$.
On the other hand, if $G$ is an ordered abelian group, then we will write $S(G)$ to denote the state space of $G$.
We will write cp and cpc for completely positive and completely positive contractive, respectively. 
Sometimes we write $a \approx_\epsilon b$ to denote $\|a - b \| < \epsilon$.

	\subsection{Sequence algebras}
	
Throughout the paper, we will often go back and forth between approximate sentences and exact ones. One of the key tools to encode approximate behaviour is the use of sequence algebras. Given a sequence $(B_n)_{n \in \mathbb{N}}$ of $\C$-algebras, the induced \emph{sequence algebra} is the quotient of
the bounded sequences by the null sequences. This will be denoted by $\prod_{n=1}^\infty B_n/\bigoplus_{n=1}^\infty B_n$ or $\prod\limits_{\infty}B_n$ for simplicity of the notation.
In the special, but relevant case where $B_n = B$ for all $n$, it is common to denote this sequence algebra by $B_\infty$. We will then view its elements as given by sequences with entries in $B$ and there is a canonical embedding $\iota_B:B\to B_\infty$ obtained by viewing each element in $B$ as a constant sequence in $B_\infty$. 

Moreover, for a sequence algebra $\prod\limits_{\infty}B_n$, we denote by $T_\infty\Big(\prod\limits_{\infty}B_n\Big)$ the set of \emph{limit traces} on $\prod\limits_{\infty}B_n$, i.e.\ those $\tau\in T\Big(\prod\limits_{\infty}B_n\Big)$ of the form $\tau((b_n)_{n \in \mathbb{N}})=\lim_{n\to\omega}\tau_n(b_n)$ for some sequence of traces $\tau_n\in T(B_n)$ and a free ultrafilter $\omega$ on $\mathbb{N}$. 

We will often want to transfer properties of the algebras $(B_n)_{n\in\mathbb{N}}$ to their sequence algebra. In particular, one fact we will need is that real rank zero passes to sequence algebras. A unital $\C$-algebra $A$ has \emph{real rank zero} if the invertible self-adjoint elements are dense among all self-adjoint elements. If $A$ is non-unital, we say it has real rank zero if its minimal unitisation has real rank zero \cite[Section 1]{brownpedersen}.

\begin{lemma}\label{lem:ultraproducts.real.ranz.zero}
	Let $(B_n)_{n\in \mathbb{N}}$ be a sequence of $\C$-algebras with real rank zero. Then $\prod\limits_{\infty} B_n$ has real rank zero.\footnote{This result actually holds for any filter on $\mathbb{N}$.}
\end{lemma}

	\begin{proof}
	We will heavily rely on \cite[Theorem 2.6]{brownpedersen} and observe that by \cite[Lemma 1.11]{CE} we can assume that each $B_n$ is unital. Let $x,y \in \prod\limits_{\infty} B_n$ be positive orthogonal contractions. Say the sequences $(a_n)_{n \in \mathbb{N}}$ and $(b_n)_{n \in \mathbb{N}}$ in $\prod\limits_{n=1}^\infty B_n$ represent $x$ and $y$, respectively.\ Let $\epsilon>0$ and $N \in \N$ such that 
	$\|a_n b_n\| < \epsilon^2$ for all $n>N$.
	Since each $B_n$ has real rank zero, \cite[Theorem 2.6]{brownpedersen} yields the existence of a projection $p_n \in B_n$ such that 
	\begin{equation}
		\notag
		\|(1-p_n)a_n\| < \epsilon \qquad \text{and} \qquad \|p_n b_n\| < \epsilon \qquad \text{with} \quad n>N.
	\end{equation}
	Set $p_n := 0$ if $n\leq N$.
	It then follows that the sequence $(p_n)_{n \in \mathbb{N}}$ induces a projection $p\in \prod\limits_{\infty} B_n$ such that 
	\begin{equation}
		\| (1-p) x \| \leq \epsilon \qquad \text{and} \qquad \| py \| \leq \epsilon. \notag
	\end{equation}	
	Again, by \cite[Theorem 2.6]{brownpedersen}, $\prod\limits_{\infty} B_n$ has real rank zero.
\end{proof}

\subsection{Nuclear dimension}
	
Before defining finite nuclear dimension for $^*$-homomorphisms, we need to recall the definition of order zero maps. A cp map $\phi: A\to B$ between $\C$-algebras is said to be
\emph{order zero} if, for every $a,b \in A_+$ with $ab=0$, one has $\phi(a)\phi(b)=0$. The general theory for these maps was developed in \cite{orderzero}.

\begin{defn}[\cite{nucdim,decomprankZstable}]\label{defdim}
	Let $A$ and $B$ be $\C$-algebras with $A$ separable and $\theta:A\to B$ a $^*$-homomorphism. It is said that $\theta$ has \emph{nuclear dimension at most $n$}, for some $n\in\mathbb{N}$ and denoted by $\dimnuc \theta\leq n$,  if for any $k\in \mathbb{N}$, there exist $F_k=F_k^{(0)}\oplus \ldots \oplus F_k^{(n)}$ finite dimensional $\C$-algebras and  maps $\psi_k:A\to F_k$ and $\eta_k:F_k\to B$ such that 

\begin{enumerate}
    \item $\|\eta_k\circ\psi_k(a)-\theta(a)\| \to 0$ for all $a\in A$,
    \item $\psi_k$ is cpc,
    \item if we denote by $\eta_k^{(i)}$ the restriction of $\eta_k$ to $F_k^{(i)}$, then $\eta_k^{(i)}$ is a cpc order zero map for all $i=0,\ldots, n$.
\end{enumerate}
If additionally each $\eta_k$ is contractive, then it is said that $\theta$ has \emph{decomposition rank} at most $n$.
\end{defn}

The $^*$-homomorphism $\theta$ has nuclear dimension equal to $n$ if $n$ is the minimum number for which we can find $(F_k,\psi_k,\eta_k)_{k\in\mathbb{N}}$ as above. In this case, we call $(F_k,\psi_k,\eta_k)_{k\in\mathbb{N}}$ an \emph{$n$-decomposable approximating system} for $\theta$.\ If no such $n$ exists, then $\theta$ has infinite nuclear dimension.\ In fact, in the case when a $^*$-homomorphism $\theta$ is unital and has nuclear dimension equal to zero, we can say more about the approximating system given by Definition \ref{defdim}. 
Essentially by definition, in this situation $\theta$ also has decomposition rank equal to zero. As in the case of finite decomposition rank for $\C$-algebras {\cite[Proposition 5.1]{decomprank}}, we can assume that the downward approximating maps are approximately multiplicative. We will record this result for later use.

\begin{lemma}{\cite[Proposition 1.7]{nucdimOinftystable}}\label{approxmultmaps}
Let $\theta:A\to B$ be a $^*$-homomorphism between $\C$-algebras with $A$ separable and $\dimnuc\theta=0$. Then there exists an $n$-decomposable approximating system
$(F_k,\psi_k,\eta_k)_{k\in\N}$ for $\theta$ such that the maps $\psi_k$ are approximately multiplicative, i.e.
the induced map $\psi\colon A\rightarrow\prod\limits_{\infty}F_k$ is a $^*$-homomorphism.
\end{lemma}

\begin{rmk}\label{remark:cnsqces.def.nuc.dim}
	
	\begin{enumerate}[(i)]
	    \item By the definition in \cite{nucdim}, $\dimnuc A= \dimnuc \id_A$ for any separable $\C$-algebra $A$.\label{rmk: NucDimitem1}
		
		\item If we take $\theta:A\to B$ to be a unital $^*$-homomorphism with finite nuclear dimension, a standard simplification in the spirit of {\cite[Remark 5.2 (ii)]{decomprank}} allows us to assume that the downward approximating maps $\psi_k$ can be taken to be unital and completely positive. In particular, if $\dimnuc\theta= 0$, we can take the downward approximating maps to be ucp and approximately multiplicative.\label{rmk: NucDimitem2}

	    \item Essentially the same proof as in {\cite[Proposition 2.3 (ii)]{nucdim}} shows that if $\theta:A\to B$ is a $^*$-homomorphism with $\dimnuc\theta=0$, then any matrix amplification of $\theta$, denoted by $\theta\otimes\id_{M_n(\mathbb{C})}$, has nuclear dimension equal to zero. This will be used in the proof of Proposition \ref{K1map}.\label{rmk: NucDimitem3}
	    
	    \item If $\phi,\psi:A\to B$ are approximately unitarily equivalent $^*$-homo-morphisms, then it is readily seen that $\dimnuc\phi=\dimnuc\psi$. This will be freely used throughout this paper.\label{rmk: NucDimitem4}
	\end{enumerate}
	\end{rmk}
	
\begin{lemma}\label{bounddimmorphism}
Let $A$ and $B$ be separable $\C$-algebras and $\theta:A\to B$ a $^*$-homomorphism. Then $\dimnuc\theta\leq \min(\dimnuc A,\dimnuc B)$.
\end{lemma}

\begin{proof}
If neither $A$ nor $B$ have finite nuclear dimension, the conclusion follows. If $\dimnuc A<\infty$, since $\theta=\theta\circ\id_A$, we can use an approximating system of $\id_A$ to get one for $\theta$. Note that composing with a $^*$-homomorphism preserves the fact that the upward maps are order zero. Hence, $\dimnuc \theta\leq \dimnuc A$. Similarly, $\dimnuc\theta\leq\dimnuc B$.
\end{proof}

\begin{rmk}
In particular, combining the lemma above and {\cite[Theorem B, Corollary C]{CPOU}}, if $A$ or $B$ are simple, unital, nuclear, and $\mathcal{Z}$-stable, then $\dimnuc\theta\leq 1$.  
\end{rmk}
	
The next proposition follows in the spirit of {\cite[Proposition 3.4]{nucdim}} and it says that the matrix blocks in the approximating finite dimensional $\C$-algebras can be taken to be large, as long as the ranks of the irreducible representations of our domain are large. In particular, if the domain has no irreducible finite dimensional representations, the sizes of the matrix blocks go to infinity.
	
\begin{prop} [cf.\ {\cite[Proposition 3.4]{nucdim}}]\label{prop: largematrixblocks}
Let $A$ and $B$ be separable $\C$-algebras and $\theta: A \to B$ a $^*$-homomorphism with $\dimnuc \theta \leq n < \infty$. Let $r\in\mathbb{N}$ and suppose $A$ has no irreducible representations of rank strictly less than $r$.
Then there exists an $n$-decomposable approximating system $(F_k, \psi_k, \eta_k)_{k\in\mathbb{N}}$
such that the irreducible representations of each $F_k$ have rank at least $r$.
\end{prop}

\begin{proof}
Let $(\bar{F}_k, \bar{\psi}_k, \bar{\eta}_k)_{k\in\mathbb{N}}$ be an approximating system witnessing nuclear dimension at most $n$ as in Definition \ref{defdim}. By {\cite[Proposition 1.7]{nucdimOinftystable}}, we can further assume that the maps $\psi_k$ are approximately order zero i.e. $$\|\psi_k(a)\psi_k(b)\| \to 0$$ for all $a,b \in A_+$ that satisfy $ab = 0$. For each $k$, we write $\bar{F}_k=F_k \oplus \tilde{F}_k$, where $\tilde{F}_k$ consists of those matrix blocks of $\bar{F}_k$ with rank at most $r-1$. Likewise, we denote by $\psi, \tilde{\psi}, \eta,$ and $\tilde{\eta}$ the respective components of $\bar{\psi}$ and $\bar{\eta}$.

Consider the cpc order zero map induced by the sequence $(\tilde{\psi}_k)_{k\in\N}$:
$$ 
\tilde{\psi}: A \to \prod_\infty \tilde{F}_k.
$$ 
Let $h\in A$ be a normalised strictly positive element and set 
$$ 
\mu := \limsup\limits_{k\to\infty}\|\tilde{\psi}_k(h)\|= \|\tilde{\psi}(h)\|.\footnote{A positive element $a\in A$ is strictly positive if $\overline{aAa}=A$.}
$$

Using the fact that the product $\prod_{k=1}^\infty \tilde{F}_k$ is $(r-1)$-subhomogeneous, there exists an irreducible representation 
$$
\pi: \prod_\infty \tilde{F}_k
\to \M_l
$$ for some $l\leq r-1$ such that $\|\pi \circ \tilde{\psi}(h)\|= \mu$.

Since $\pi$ is a $^*$-homomorphism, $\pi\circ \tilde{\psi}$ is a cpc order zero map. 
By the structure theorem for order zero maps (\cite[Theorem 3.3]{orderzero}), there exist a $^*$-homomorphism $\sigma :A \to \M_l$ and $0\leq d\leq 1_{\M_l} \in \M_l$ such that 
$$
\pi\circ\tilde{\psi}(a)= d\sigma(a)=\sigma(a)d
$$ 
for any $a\in A$. 
However, by hypothesis $A$ has no irreducible representations of rank strictly less than $r$, so $\sigma = 0$. 
Thus
$$
\|\tilde{\psi}(h)\|=\mu= \|\pi\tilde{\psi}(h)\|=0.
$$ 
Since $\tilde{\psi}$ is  cpc order zero map and $h$ is a strictly positive element, we can conclude that $\tilde{\psi}=0$. Hence, $(F_k, \psi_k, \eta_k)_{k\in\mathbb{N}}$ is a system of cp approximations for $\theta$ with the required properties.
\end{proof}

\begin{cor}\label{largematrixblocks}
	
	Let $A$ and $B$ be $\C$-algebras with $A$ simple separable and non-elementary.\footnote{A $\C$-algebra $A$ is non-elementary if it is not the $\C$-algebra of compact operators on a Hilbert space.} Suppose $\theta: A \to B$ is a $^*$-homomorphism with finite nuclear dimension. Then there is a finite dimensional approximating system $(F_k, \psi_k, \phi_k)_{k \in \mathbb{N}}$ witnessing the finite nuclear dimension of $\theta$ where each $F_k$ has no irreducible representations of rank less than $k$; i.e.\ all matrix blocks in the decomposition of $F_k$ have size at least $k$.
\end{cor}

\subsection{The $\underline{K}T_u$-invariant}	

We will collect some facts we need about the total invariant. We refer the reader to Sections $2$ and $3$ in \cite{classif} for a detailed description.

Let $A$ be a unital $\C$-algebra. The \emph{total $K$-theory} of $A$, $\underline{K}(A)$, is the combination of $K_*(A)$
together with $\bigoplus_{n\geq 2}K_*(A;\mathbb{Z}/n)$ and natural maps between them. There are multiple equivalent ways of defining the groups $K_i(A;\mathbb{Z}/n)$, but in this paper we choose to say that $K_i(A;\mathbb{Z}/n):=K_i(A\otimes\mathcal{O}_{n+1})$ for $i=0,1$.\footnote{Here $\mathcal{O}_{n+1}$ denotes the Cuntz algebra generated by $n+1$ isometries for $n\geq 2$.} These all fit into natural commuting diagrams with the Bockstein maps (see for example {\cite[Section 2]{classif}}). We will use repeatedly the following short exact sequences 
\begin{equation}\label{Kthcoeffses}
\begin{tikzcd}
0 \ar{r} & K_i(A)\otimes \mathbb{Z}/n \ar{r} & K_i(A;\mathbb{Z}/n) \ar{r} & \Tor(K_{1-i}(A), \mathbb{Z}/n) \ar{r} & 0,
\end{tikzcd}
\end{equation} for $i=0,1$.

In general, for unital $\C$-algebras  $A$ and $B$, a \emph{$\Lambda$-morphism} $$\underline{\alpha}:\underline{K}(A)\to\underline{K}(B)$$ consists of $^*$-homomorphisms $$\alpha_i : K_i(A)\to K_i(B) \qquad \text{and} \qquad \alpha_i^{(n)} : K_i(A;\mathbb{Z}/n)\to K_i(B;\mathbb{Z}/n)$$ for all $i=0,1$ and $n\geq 2$ intertwining all the Bockstein operations.\ In particular, we will use that the diagram
\begin{equation}\label{eq: Bockstein}
\begin{tikzcd}
K_0(A;\mathbb{Z}/n) \ar{r}{\nu_{0,A}^{(n)}}\ar{d}{\alpha_0^{(n)}} & K_1(A) \ar{d}{\alpha_1}\\
K_0(B;\mathbb{Z}/n) \ar{r}{\nu_{0,B}^{(n)}} & K_1(B),
\end{tikzcd}
\end{equation} commutes.\footnote{Note that $\nu_{0,A}^{(n)}$ and $\nu_{0,B}^{(n)}$ denote canonical Bockstein maps.} The collection of these $\Lambda$-morphisms is denoted $\Hom_{\Lambda}(\underline{K}(A),\underline{K}(B))$.\ For a unital $^*$-homomorphism $\theta:A\to B$, we denote the induced $\Lambda$-morphism by $\underline{K}(\theta):\underline{K}(A)\to\underline{K}(B)$.

In any unital $\C$-algebra, the traces are related to the $K$-theory via the canonical pairing map. Precisely, let $\rho_A: K_0(A)\to \Aff(T(A))$ be defined by $\rho_A([x]_0)(\tau)=K_0(\tau)(x)$ for any $[x]_0\in K_0(A)$ and any $\tau\in T(A)$. Moreover, for a $^*$-homomorphism $\theta$, we will denote by $\Aff(T(\theta)):\Aff(T(A))\to \Aff(T(B))$ the map $\Aff(T(\theta))(f)(\tau)=f(\tau\circ\theta)$ for any $f\in\Aff(T(A))$ and any $\tau\in T(B)$.

Given a unital $\C$-algebra $A$, we can associate an abelian group called the \emph{(Hausdorffised) unitary algebraic $K_1$-group}, denoted by $\overline{K_1}^{\alg}(A)$, in a functorial way. 
In particular, any unital $^*$-homomorphism $\theta:A\to B$ induces a group homomorphism $\overline{K_1}^{\alg}(\theta):\overline{K_1}^{\alg}(A)\to \overline{K_1}^{\alg}(B)$. 

Importantly, we can relate $\overline{K_1}^{\alg}(A)$ with $\Aff(T(A))$ and $K_1(A)$ via canonical maps. There is a map $\cancel{a}_A: \overline{K_1}^{\alg}(A) \to K_1(A)$ that, roughly speaking, sends a class in $\overline{K_1}^{\alg}(A)$ to its corresponding class in $K_1(A)$. Moreover, there is a map $\Th_A:\Aff(T(A))\to \overline{K_1}^{\alg}(A)$ called the Thomsen map. We refer to \cite{thomsenmap} for a construction of the Hausdorffised unitary algebraic $K_1$-group and of the Thomsen map. One fact we are going to use repeatedly is that there is a non-canonical decomposition
\begin{equation}\label{algK1decomp}
\overline{K_1}^{\alg}(A)\cong K_1(A)\oplus \frac{\Aff(T(A))}{\overline{\rho_A(K_0(A))}}, 
\end{equation}{\cite[Corollary 3.3]{thomsenmap}}.\footnote{To see that this decomposition is non-canonical, the reader could consult \cite{NT96} or \cite[Example 9.11]{classif}.}

As it was worked out in {\cite[Section 3]{classif}}, the last ingredient needed in the total invariant to obtain the classification of morphisms are certain natural maps $$\zeta_A^{(n)}:K_0(A;\mathbb{Z}/n)\to \overline{K_1}^{\alg}(A), \quad n\geq 2.$$ Let $\beta:\overline{K_1}^{\alg}(A)\to \overline{K_1}^{\alg}(B)$ be a group homomorphism which is compatible with the Thomsen map (see \eqref{KthTracepairing}).\ To get a compatible morphism between the total invariants of two $\C$-algebras $A$ and $B$, we will further require that the following diagram commutes:

\begin{equation}\label{newpairingdiagram1}
\begin{tikzcd}
K_0(A;\mathbb{Z}/n) \ar{r}{\zeta_A^{(n)}}\ar{d}{\alpha_0^{(n)}} & \overline{K_1}^{\alg}(A) \ar{d}{\beta}\\
K_0(B;\mathbb{Z}/n) \ar{r}{\zeta_B^{(n)}} & \overline{K_1}^{\alg}(B).
\end{tikzcd}
\end{equation} 

It was shown in {\cite[Proposition 3.12]{classif}} that the diagram in \eqref{newpairingdiagram1} always commutes if $A$ and $B$ are unital $\C$-algebras with $K_1(A)$ torsion free.

\begin{defn}[{\cite[Definition 3.5]{classif}}]\label{KTuInv}
Let $A,B$ be unital $\C$-algebras. The \emph{total invariant} of $A$, denoted by $\underline{K}T_u(A)$, is 
\begin{equation}
    \underline{K}T_u(A):=(\underline{K}(A), \Aff(T(A)), \overline{K_1}^{\alg}(A), [1_A]_0, \rho_A, \Th_A, \cancel{a}_A, (\zeta_A^{(n)})_{n\geq 2}).
\end{equation} 

A morphism $(\underline{\alpha}, \beta, \gamma): \underline{K}T_u(A)\to \underline{K}T_u(B)$ consists of a $\Lambda$-morphism $\underline{\alpha}\in \Hom_{\Lambda}(\underline{K}(A),\underline{K}(B))$ with $\alpha_0([1_A]_0)=[1_B]_0$, a positive unital linear map $\gamma: \Aff(T(A))\to \Aff(T(B))$, and a group homomorphism $\beta:\overline{K_1}^{\alg}(A)\to \overline{K_1}^{\alg}(B)$ such that the diagram

\begin{equation}\label{KthTracepairing}
\begin{tikzcd}
K_0(A) \ar{r}{\rho_A}\ar{d}{\alpha_0} & \Aff(T(A)) \ar{r}{\Th_A}\ar{d}{\gamma} & \overline{K_1}^{\alg}(A) \ar{r}{\cancel{a}_A}\ar{d}{\beta} & K_1(A)\ar{d}{\alpha_1}\\
K_0(B) \ar{r}{\rho_B} & \Aff(T(B)) \ar{r}{\Th_B} & \overline{K_1}^{\alg}(B) \ar{r}{\cancel{a}_B} & K_1(B)
\end{tikzcd}
\end{equation} and the diagram in \eqref{newpairingdiagram1} commute. In particular, by  {\cite[Proposition 3.6]{classif}}, any unital $^*$-homomorphism $\theta:A\to B$ induces a natural $\underline{K}T_u$-morphism  given by $$\underline{K}T_u(\theta) :=(\underline{K}(\theta), \overline{K_1}^{\alg}(\theta), \Aff(T(\theta))): \underline{K}T_u(A)\to \underline{K}T_u(B).$$ 
\end{defn}

Moreover, it is said that $(\underline{\alpha}, \beta, \gamma): \underline{K}T_u(A)\to \underline{K}T_u(B)$ is \emph{faithful} if the induced map $\gamma^*:T(B)\to T(A)$ (which exists by Kadison's duality \cite{Kad51}) satisfies that $\gamma^*(\tau)$ is faithful for all $\tau\in T(B)$. Note that this is automatic if $A$ is simple, which is the case in this paper. 
The main classification theorem we will use in this paper is the following.

\begin{theorem}[{\cite[Theorem 9.3]{classif}}]\label{classifthm}
Let $A$ and $B$ be unital separable nuclear $\C$-algebras such that $A$ satisfies the UCT and $B$ is simple and $\Z$-stable. If $(\underline{\alpha}, \beta, \gamma): \underline{K}T_u(A) \to \underline{K}T_u(B)$ is a faithful morphism, then there exists a unital full $^*$-homomorphism $\phi: A \to B$ such that $\underline{K}T_u(\phi) = (\underline{\alpha},\beta,\gamma)$, and this $\phi$ is unique up to approximate unitary equivalence.
\end{theorem}

We are going to use repeatedly the following special case of the classification theorem above.

\begin{lemma}\label{lemma : classifthmrr0}
Suppose $A$ and $B$ are $\C$-algebras as in Theorem \ref{classifthm} and suppose further that $B$ has real rank zero. Then, for any triple $(\underline{\alpha}, \beta, \gamma)$ satisfying the diagram in \eqref{KthTracepairing}, $\beta$ is determined by $\alpha_1$, and the new pairing map diagram in \eqref{newpairingdiagram1} is automatically satisfied. Therefore, unital full $^*$-homomorphisms $\phi:A\to B$ are classified by total $K$-theory and traces.
\end{lemma}

\begin{proof}
If $B$ is purely infinite, then \eqref{algK1decomp} shows that $\overline{K_1}^{\alg}(B)\cong K_1(B)$. If $B$ is stably finite, since $B$ is simple, nuclear, $\mathcal{Z}$-stable, and has real rank zero, it follows that the image of $K_0(B)$ under the pairing map is uniformly dense in $\Aff(T(B))$ (\cite[Theorem 7.2]{rr0char}). Therefore, $\overline{K_1}^{\alg}(B)\cong K_1(B)$. In both cases, the isomorphism is given by $\cancel{a}_B$, so the diagram in \eqref{KthTracepairing} shows that $\beta=\cancel{a}_B^{-1}\circ \alpha_1\circ \cancel{a}_A$. Therefore, $\beta$ is determined by the $K$-theory maps and traces.

It remains to check that the diagram in \eqref{newpairingdiagram1} commutes. By {\cite[Proposition 3.2]{classif}}, the diagram
\begin{equation}\label{newpairingdef}
\begin{tikzcd}
K_0(A;\mathbb{Z}/n) \ar{r}{\nu_{0,A}^{(n)}}\ar{d}{\zeta_A^{(n)}} & K_1(A) \ar{d}{\id}\\
\overline{K_1}^{\alg}(A) \ar{r}{\cancel{a}_A} & K_1(A),
\end{tikzcd}
\end{equation} commutes. 
Since $B$ has real rank zero, $$\beta\circ\zeta_A^{(n)}= \cancel{a}_B^{-1}\circ \alpha_1\circ \cancel{a}_A\circ \zeta_A^{(n)}.$$ Using diagram \eqref{newpairingdef}, it follows that $\beta\circ\zeta_A^{(n)}= \cancel{a}_B^{-1}\circ \alpha_1\circ \nu_{0,A}^{(n)}$. Then, $\alpha_1\circ \nu_{0,A}^{(n)}=\nu_{0,B}^{(n)}\circ \alpha_0^{(n)}$ by \eqref{eq: Bockstein}. Finally, using diagram \eqref{newpairingdef} for $B$, it follows that $\zeta_B^{(n)}=\cancel{a}_B^{-1}\circ\nu_{0,B}^{(n)}$. Thus 
\begin{align}
\beta\circ\zeta_A^{(n)}=\cancel{a}_B^{-1}\circ \nu_{0,B}^{(n)}\circ \alpha_0^{(n)}=\zeta_B^{(n)}\circ \alpha_0^{(n)}.
\end{align}
Therefore, the diagram in \eqref{newpairingdiagram1} commutes and the conclusion follows by Theorem \ref{classifthm}.
\end{proof}

\section{Products of finite dimensional $\C$-algebras}\label{ultraprod}

One goal of this section is to show that the $K_0$-group of a sequence algebra of finite dimensional $\C$-algebras can be thought of as a dimension group which is not necessarily countable. We will state some standard results concerning dimension groups, but we refer the reader to \cite{goodearlbook} for details.\footnote{Note that in \cite{goodearlbook} a dimension group can be uncountable, but we are using the convention from \cite{rordambook}.}

\begin{defn}
A countable ordered abelian group $(G, G_+)$ is called a \emph{dimension group} if it is isomorphic to the inductive limit of a sequence

\[
\begin{tikzcd}
\mathbb{Z}^{r_1} \ar{r}{\alpha_1} & \mathbb{Z}^{r_2} \ar{r}{\alpha_2} & \mathbb{Z}^{r_3} \ar{r}{\alpha_3} & \ldots & ,
\end{tikzcd}
\] for some natural numbers $r_n$, some positive group homomorphisms $\alpha_n$, where $\mathbb{Z}^r$ is equipped with its standard ordering given by $$(\mathbb{Z}^r)_+=\{(x_1,x_2,\ldots, x_r) : x_j\geq 0\}.$$ 
\end{defn}

In general, if $(G,G_+)$ is an ordered group, $G$ is said to be \emph{simple} if every non-zero positive element is an order unit. We further record some important structural properties of ordered abelian groups. An ordered abelian group $G$ is called \emph{unperforated} if for any $n\in\mathbb{N}$ and any $g\in G$ such that $ng\in G_+$, then $g\in G_+$. Also, $G$ is said to have the \emph{Riesz interpolation property} if for any $g_1,g_2,h_1,h_2\in G$ such that $g_i\leq h_j$ for any $i,j=1,2$ there exists $z\in G$ such that $g_i\leq z\leq h_j$ for any $i,j=1,2$.\ An ordered abelian group $(G,G_+)$ is said to have the \emph{Riesz decomposition property} if for any $g,h_1,h_2 \in G_+$ such that $g \leq  h_1 + h_2$, there exist $g_1,g_2 \in G_+$ such
that $g = g_1 + g_2$ and $g_j \leq h_j$ for each $j=1,2$. In fact, the Riesz interpolation property is equivalent to the Riesz decomposition property for ordered abelian groups (see \cite[Proposition 2.1]{goodearlbook}).

Using a classical theorem of Effros, Handelman, and Shen (\cite[Theorem $2.2$]{dimgroup}), a countable ordered abelian group is a dimension group if and only if it is \emph{unperforated} and has the \emph{Riesz interpolation property}. In fact, all dimension groups are realised as the $K_0$-group of some AF-algebra.

\begin{prop}[{\cite[Proposition 1.4.2, Corollary 1.5.4]{rordambook}}]\label{AFdimgroup}
The ordered $K_0$-group of an $AF$-algebra is a dimension group and for every dimension group $G$, there is an $AF$-algebra $A$ such that $K_0(A)$ and $G$ are isomorphic as ordered abelian groups. Moreover, if the dimension group $G$ is simple, then the corresponding $AF$-algebra is simple.
\end{prop}

Since we will be mostly interested in dimension groups coming from $\mathcal{Z}$-stable AF-algebras, finite dimensional representations are a potential issue. However, the lack of finite dimensional representations can be encoded at the level of the $K_0$-group. For this, it is said that an ordered abelian group $G$ has \emph{property $(D)$} if for each order unit $x$ in $G$, there exists an order unit $y$ in $G$ such that $2y\leq x$. As shown in {\cite[Lemma 9]{embZ}}, the $K_0$-group of a unital AF-algebra $A$ has \emph{property $(D)$} if and only if $A$ has no non-zero finite dimensional representations. 
Let $(F_k)_{k \in \mathbb{N}}$ be a sequence of finite dimensional $\C$-algebras and denote by $E:= \prod\limits_{\infty}F_k$ the sequence algebra induced by them. Our key observation is that $K_0(E)$ is almost the $K_0$-group of an AF-algebra, with the mention that it might not be countable.

Let us also recall the construction of the Cuntz semigroup of $E$. If $a,b\in E\otimes \mathcal{K}$ are positive, we will write $a\precsim b$ to mean that there exists a sequence $(x_n)_{n\in\mathbb{N}}\subset E\otimes \mathcal{K}$ such that $x_nbx_n^*\to a$ as $n\to\infty$. We write $a\sim b$ if $a\precsim b$ and $b\precsim a$ and the Cuntz semigroup of $E$ is $(E\otimes\mathcal{K})_+/\sim$. The Cuntz semigroup of $E$ is \emph{unperforated} if for any $n\in\mathbb{N}$ and $a,b\in (E\otimes\mathcal{K})_+$ such that $a\otimes 1_n\precsim b\otimes 1_n$, we have $a\precsim b$.

\begin{prop}\label{K0ultraprod}
Let $E$ as defined above. Then $K_0(E)$ is an unperforated ordered abelian group which satisfies the Riesz interpolation property. Moreover, if for all $k\geq 1$, $F_k$ has no irreducible representations of rank less than $k$, then $K_0(E)$ satisfies property $(D)$.
\end{prop}

\begin{proof}
By Lemma \ref{lem:ultraproducts.real.ranz.zero}, $E$ has real rank zero.
Then, \cite[Corollary 1.3]{riesz} shows that projections in the $K_0$-group satisfy the Riesz decomposition property with respect to the Murray-von Neumann equivalence. Moreover, since $E$ is unital and has stable rank one (see for example {\cite[Lemma $1.21$]{bbstww}})\footnote{Technically this was proved for $\prod\limits_{\omega}F_k$, but the same proof works for the sequence algebra $\prod\limits_{\infty}F_k$.}, it has cancellation (\cite[Proposition 6.5.1]{blackadar}) and hence its $K_0$-group satisfies the Riesz decomposition property. This condition is equivalent to the Riesz interpolation property by {\cite[Proposition 2.1]{goodearlbook}}. 

Now, let $n\in \mathbb{N}$ and $x\colon=[p]_0-[q]_0\in K_0(E)$ such that $nx\in K_0(E)_+$. This implies that $n[q]_0\leq n[p]_0$ and since we have cancellation, this is equivalent to saying that $q\otimes 1_n\precsim p\otimes 1_n$ in the Cuntz semigroup of $E$. Since the Cuntz semigroup of $E$ is unperforated (\cite[Lemma 2.3(ii)]{CuntzUnperf}), we get that $q\precsim p$, which is equivalent to $x\in K_0(E)_+$.\ Hence, $K_0(E)$ is unperforated.

For the last part, since we have strict comparison with respect to limit traces (essentially by the same proof of {\cite[Lemma 1.23]{bbstww}}), order units in $K_0(E)$ correspond precisely to those elements which are strictly positive on all limit traces.  
Let us choose an order unit $[p]_0
 \in K_0(E)$, which is induced by a sequence of projections $p_k\in \mathbb{M}_N(F_k)$ for some natural number $N$. Then, if $F_k$ has a decomposition into $r$ matrix blocks, we can write $p_k$ as a sum of orthogonal projections $p_k^{(1)}\oplus\ldots\oplus p_k^{(r)}$.\footnote{Note that $r$ depends on $k$, but we supress this for ease of notation.}

Since $p$ is strictly positive on all limit traces and the trace space is compact, there exists some $n\in\mathbb{N}$ such that $\tau(p)\geq 1/n$ for all limit traces $\tau$. 
By hypothesis, each of the blocks of $F_k$ has size at least $k$, and denote the size of the block $i$ by $n^{(k)}_i$. Then, we can consider $q=(q_k)_{k \in \mathbb{N}}$, where $q_k=q_k^{(1)}\oplus \ldots \oplus q_k^{(r)}$, and each $q_k^{(i)}$ is a diagonal matrix with $\lfloor \frac{n^{(k)}_i}{3n} \rfloor$ entries equal to $1$ on the main diagonal and the rest equal to $0$. Then, for any trace $\tau$ on $F_k$, one has that $2\tau(q_k)<\frac{1}{n}\leq \tau(p)$, where we still denote by $\tau$ the non-normalised extension to $\mathbb{M}_N(F_k)$. Since each $F_k$ has strict comparison, it follows that $2q_k \precsim p_k$, so $2[q]_0 \leq [p]_0$. Note that $[q]_0$ constructed in this way is an order unit. Since $n_i^{(k)}\geq k$ for all $i$ and all $k\in\mathbb{N}$, $q_k$ is nonzero for $k$ large enough. Moreover, the trace of $q_k^{(i)}$ is $\lfloor \frac{n^{(k)}_i}{3n} \rfloor/n^{(k)}_i$, which is at least $\frac{1}{4n}$ for $k$ large enough. Thus, $q$ is strictly positive on all limit traces, so $[q]_0$ is an order unit. Hence the conclusion follows.
\end{proof}

Therefore, the lack of countability is the only obstruction that prevents us to say that $K_0(E)$ is a dimension group. However, all of these abstract properties defining a dimension group together with property $(D)$ are \emph{separably inheritable} in the sense of Blackadar {\cite[Definition II.8.5.1]{sepinheritblackadar}}. Thus, we can proceed to find a separable $\C$-subalgebra of the sequence algebra such that its $K_0$-group is a dimension group. Let us first recall the definition of a separably inheritable property.

\begin{defn}{\cite[Definition II.8.5.1]{sepinheritblackadar}}\label{defn: SepInh}
A property $P$ of $\C$-algebras is \emph{separably inheritable} if the following hold:

\begin{enumerate}
    \item Whenever $A$ is a $\C$-algebra with property $P$ and $B$ is a separable $\C$- subalgebra of $A$, then there is a separable $\C$-subalgebra $C$ of $A$ which contains $B$ and which has property $P$;
    \item whenever $(A_n,\phi_n)$ is an inductive system of separable $\C$-algebras with injective connecting maps, and each $A_n$ has property $P$, so does $\lim\limits_{\longrightarrow} A_n.$
\end{enumerate}
\end{defn}

\begin{prop}[{\cite[II.8.5.4, II.8.5.5]{sepinheritblackadar}}]\label{sepinh}
The following properties are separably inheritable:

\begin{enumerate}
    \item $A$ has real rank zero;
    \item $A$ has stable rank one;
    \item $K_0(A)$ is unperforated;
    \item $K_0(A)$ has the Riesz interpolation property;
    \item $K_1(A)=0$.
\end{enumerate}
\end{prop}

Note that property $(D)$ can be defined in the general context of ordered abelian groups and we are going to show that the fact that $K_0(A)$ has property $(D)$ is a separably inheritable property. Note that by passing to unitisations, it is enough to consider the unital case.

\begin{lemma}\label{Dsepinh}
Let $A$ be a unital $\C$-algebra such that $K_0(A)$ has property $(D)$ and let $S\subset A$ be a countable subset. Then, there exists a separable $\C$-subalgebra $B$ of $A$, containing $S$, such that $K_0(B)$ has property $(D)$. Moreover, having a $K_0$-group with property $(D)$ is a separably inheritable property.
\end{lemma}

\begin{proof}
Since $S$ is countable, we can assume that $1_A\in S$. We can write $S=\{s_0,s_1,\ldots\}$ with $s_0=1_A$. 
Set $S_0:=\{1_A\}$ and $B_0:=\C(S_0)$. Since $K_0(A)$ has property $(D)$, there exists a projection $y_0$ in some matrix amplification of $A$ such that $2[y_0]\leq[1_A]$, $[y_0]$ is an order unit, and the subequivalence is realised by some partial isometry $v_0$. Then let $S_1$ be the set containing $S_0,s_1$, all entries of $y_0$, and all entries of $v_0$, and let $B_1=\C(S_1)$. Then $B_1$ is a separable $\C$-algebra, so $K_0(B_1)$ is a countable ordered abelian group. Since $1_A\in B_1$ induces an order unit in $K_0(B_1)$, any order unit in $K_0(B_1)$ induces a corresponding order unit in $K_0(A)$. Then, for each order unit $[x]_0$ in $K_0(B_1)$, there is an order unit $[y_1]_0$ in $K_0(A)$ such that $2[y_1]_0\leq[x_1]_0$ in $K_0(A)$ via some partial isometry $v_1$. Performing this operation for each of the countably many order units in $K_0(B_1)$, we let $S_2$ be the set containing $S_1,s_2,$ all the entries of the chosen order units and partial isometries imported from $A$. Let $B_2=\C(S_2)$. 

Inductively, we can choose increasing countable sets $S_n$ and $B_n:=\C(S_n)$. Now take $B$ to be the closure of $\bigcup_{n\geq 1}B_n$. Then $B$ is a separable $\C$-subalgebra of $A$ containing $S$. Moreover, it has property $(D)$ by construction and as any projection in $B$ is Murray-von Neumann equivalent to a projection in some $B_n$.

Finally, condition (ii) in Definition \ref{defn: SepInh} is seen to be satisfied by continuity of the $K_0$-functor. Thus, having a $K_0$-group with property $(D)$ is a separably inheritable property.
\end{proof}

Then, provided that the sizes of irreducible representations of each $F_k$ are large enough, we can show that there is a separable $\C$-subalgebra of $\prod\limits_{\infty} F_k$ preserving a series of $K$-theoretic properties.

\begin{prop}\label{dimgroup}
	Let $A$ be a simple separable unital non-elementary $\C$-algebra and let $(F_k)_{k \in \mathbb{N}}$ be a sequence of finite dimensional $\C$-algebras.
	Suppose $ \psi:A \to \prod\limits_{\infty} F_k$ is a unital $^*$-homomorphism.
	Then, there exists a unital separable $\C$-algebra $E$ of real rank zero and stable rank one\footnote{It is said that a unital $\C$-algebra has \emph{stable rank one} if the invertible elements are dense.} such that
	\begin{enumerate}
		
		\item $\psi(A)\subset E \subset \prod\limits_{\infty} F_k$,
		
		\item $K_0(E)$ is unperforated with the Riesz interpolation property,
		
		\item $K_0(E)$ has property (D), 
		
		\item $K_1(E)=0$.
		
	\end{enumerate}
\end{prop}

\begin{proof}
Note first that $\prod\limits_{\infty}F_k$ has real rank zero by Lemma \ref{lem:ultraproducts.real.ranz.zero}, stable rank one by {\cite[Lemma 1.21]{bbstww}}; and its $K_0$-group is unperforated and satisfies the Riesz interpolation property by Proposition \ref{K0ultraprod}. 
Moreover, since $A$ is simple and non-elementary, it has no finite dimensional representations. Note that any finite dimensional irreducible representation of $\prod\limits_{\infty}F_k$ induces a finite dimensional representation of $A$. Therefore, the same proof as in Proposition \ref{prop: largematrixblocks} shows that, up to reindexing, $F_k$ has no irreducible representations of rank less than $k$. Hence $K_0\Big(\prod\limits_{\infty}F_k\Big)$ has property $(D)$ by Proposition \ref{K0ultraprod}.

Since a countable intersection of separably inheritable properties is separably inheritable ({\cite[Proposition II.8.5.3]{sepinheritblackadar}}), by Proposition \ref{sepinh} and Proposition \ref{Dsepinh}, there is a unital separable $\C$-algebra $E$ of real rank zero and stable rank one which satisfies the required conditions. 
\end{proof}

\section{Zero dimensional $^*$-homomorphisms}\label{zerodim}

In this section we will study how $^*$-homomorphisms with nuclear dimension equal to zero act on the total invariant. We will begin with the following lemma that states that for $^*$-homomorphisms of nuclear dimension zero the second map in the approximation can be assumed to be a $^*$-homomorphism rather than an order zero map. This follows in the spirit of \cite[Theorem 3.4]{Win03} and we include it for completeness.

\begin{lemma}\label{homupwardmap}
	Let $\theta:A\to B$ be a unital $^*$-homomorphism between unital separable $\C$-algebras with $\dimnuc \theta=0$. Let $\mathcal{F}\subset A$ be a finite subset and $\epsilon>0$. Then there exist a finite dimensional $\C$-algebra $F$, a ucp map $\psi:A\to F$, and a unital $^*$-homomorphism $\eta:F\to B$ such that $\|\theta(x)-\eta \circ \psi(x)\|< \epsilon$ for all $x\in\mathcal{F}$.
\end{lemma}

\begin{proof}
	We can assume that $1_A\in\mathcal{F}$ and all the elements in $\mathcal{F}$ are contractions. 
Since $\dimnuc\theta=0$, there exist a finite dimensional $\C$-algebra $F$, a ucp map  $\psi: A \to F$ and a cpc order zero map $\phi: F \to A$ such that
$\theta(x) \approx_{\epsilon/2} \phi \circ \psi(x)$
for all $x\in\mathcal{F}$.
	
By the structure theorem for order zero maps (\cite[Theorem 3.3]{orderzero}), there exists a unital $^*$-homomorphism $\eta:F\to \mathcal{M}(C^*(\phi(F))) \cap \{\phi(1_F)\}'$ such that 
	$$
	\phi(a)=\phi(1_F)\eta(a)
	$$ 
	for all $a\in F$.
	Note that since $\theta$ and $\psi$ are unital, $\phi(1_F) \approx_{\epsilon/2} 1_B$, which yields that $\phi(1_F)$ is invertible whenever $\epsilon$ is sufficiently small.
	Hence, $\eta$ can be considered as a unital $^*$-homomorphism into $B$ by noticing that 
	$$
	\eta(a) =\phi(1_F)^{-1/2}\phi(a)\phi(1_F)^{-1/2}.
	$$ 
	Moreover, for all $x\in\mathcal{F}$, 
	\begin{align*}
		\eta \circ \psi(x) \approx_{\epsilon/2} \phi(1_F) 	\eta \circ \psi(x) = \phi \circ \psi(x) \approx_{\epsilon/2} \theta(x).
\end{align*}
	This completes the proof.
\end{proof}

Now we state a sequence algebra version of the previous lemma.

\begin{lemma}\label{homsfactoring}
Let $\theta:A\to B$ be a unital $^*$-homomorphism between unital separable $\C$-algebras with $\dimnuc \theta=0$. Then there exist finite dimensional $\C$-algebras $(F_k)_{k \in  \mathbb{N}}$ and unital $^*$-homomorphisms $\psi:A\to \prod\limits_{\infty}F_k$, $\eta:\prod\limits_{\infty}F_k \to B_{\infty}$ such that $\iota_B\circ\theta=\eta\circ\psi.$ 
\end{lemma}

\begin{proof}
This follows by combining Lemma \ref{approxmultmaps}, \eqref{rmk: NucDimitem2} of Remark \ref{remark:cnsqces.def.nuc.dim}, and Lemma \ref{homupwardmap}.
\end{proof}

We now proceed to examine how $^*$-homomorphisms with nuclear dimension equal to zero act on the total invariant.

\begin{prop}\label{K1map}
		Let $\theta:A\to B$ be a unital $^*$-homomorphism between unital separable $\C$-algebras. If $\dimnuc\theta=0$, then $K_1(\theta)=0$.
\end{prop}
	
\begin{proof}
		Let $u\in A$ be a unitary, $\mathcal{F}=\{u,u^*,1_A\}\subset A$, and $\epsilon \geq \delta>0$. 
		By hypothesis, there exist a finite dimensional $\C$-algebra $F$ and cpc maps $\psi:A\to F$, $\eta:F\to B$ such that
		$$
		\|\eta(\psi(x))-\theta(x)\|<\delta/2, \qquad \qquad  x\in\mathcal{F}.
		$$
		By \eqref{rmk: NucDimitem2} of Remark \ref{remark:cnsqces.def.nuc.dim} and Lemma \ref{homupwardmap}, we can further assume that $\eta$ is a unital $^*$-homomorphism and $\psi$ is ucp and \emph{approximately multiplicative on $\mathcal{F}$ up to $\delta$}; i.e.\ $\|\psi(x) \psi(y) - \psi(xy)\| < \delta$ for all $x,y \in \mathcal{F}$. 
		
		We now proceed to prove that $[\theta(u)]_1=0$ in $K_1(B)$. Since $\psi$ is approximately multiplicative on $\mathcal{F}$ up to $\delta$, we obtain that
		\begin{align}
			\|\psi(u)\psi(u)^*-1_F\|<\delta \quad \text{and} \quad  \|\psi(u)^*\psi(u)-1_F\|<\delta. \notag 
		\end{align}
		By assuming that $\delta$ is sufficiently small, there exists a unitary $v\in F$ such that $\|\psi(u)-v\|<\epsilon/2$.
		Since the unitary group of a finite dimensional $\C$-algebra is connected, it follows that $v$ is homotopic to the unit $1_F$. This yields that $\eta(v)$ is homotopic to $\eta(1_F)=1_B$ in $B$. Moreover, $$\|\theta(u)-\eta(v)\|\leq \|\theta(u)-\eta \circ \psi(u)\|+\|\eta \circ \psi(u)-\eta(v)\| < \epsilon.$$
		
		Since we can assume $\epsilon<1$, this implies that $\theta(u)$ and $\eta(v)$ are homotopic unitaries in $B$. Thus, $\theta(u)$ is homotopic to $1_B$ and $[\theta(u)]_1=0$ in $K_1(B)$. 
		The same argument works for
		any unitary in any matrix amplification of $A$ since any matrix amplification of $\theta$ also has nuclear dimension equal to zero by \eqref{rmk: NucDimitem3} of Remark \ref{remark:cnsqces.def.nuc.dim}. Hence $K_1(\theta)=0$. 
	\end{proof}

     \begin{prop}\label{K1coeffmap}
	Let $\theta:A\to B$ be a unital $^*$-homomorphism between unital separable $\C$-algebras with $\dimnuc \theta=0$. Then, $K_1(\theta;\mathbb{Z}/n)=0$ for all $n\in\mathbb{N}$.
     \end{prop}

     \begin{proof}
	For all $n\geq 2$, let $\id_{\mathcal{O}_n}$ be the identity map on the Cuntz algebra $\mathcal{O}_n$. Observe that this proof reduces to showing that $K_1(\theta \otimes \id_{\mathcal{O}_n})=0$. This will be proved using a similar strategy as the one employed in Proposition \ref{K1map}. 
	Let $u \in A \otimes \mathcal{O}_n$ be a unitary and $\epsilon \geq \delta >0$. Then there exist $(a_i)_{i=1}^k \subset A$ and $(x_i)_{i=1}^k \subset \mathcal{O}_n$ such that 
	\begin{align}
		\left\| u - \sum_{i=1}^{k} a_i \otimes x_i \right\| < \frac{\delta}{6}.
	\end{align}
	Set $\mathcal{F}=\{a_1, \ldots, a_k \}$, $C=\max_{i,j}(\|x_ix_j^*\|,1)$, and let $(F, \psi, \eta)$ be an approximation witnessing nuclear dimension equal to zero for the map $\theta$. As before, we can further assume that $\eta$ is a $^*$-homomorphism and $\psi$ is approximately multiplicative on $\mathcal{F}$ up to $\frac{\delta}{3Ck^2}$. It follows that
	\begin{align}
		\left(\psi \otimes \id_{\mathcal{O}_n}\right) (u) \left(\psi \otimes \id_{\mathcal{O}_n}\right)(u)^* & \approx_{\delta/3} \left(\psi \otimes \id_{\mathcal{O}_n}\right) \left(\sum_{i=1}^k a_i \otimes x_i  \right) \left(\psi \otimes \id_{\mathcal{O}_n}\right) \left(\sum_{i=1}^k a_i \otimes x_i  \right)^* \notag \\
		& = \sum_{i,j=1}^k \psi(a_i) \psi( a^*_j) \otimes x_i x^*_j \notag \\
		& \approx_{\delta/3} \sum_{i,j=1}^k \psi(a_ia^*_j) \otimes x_i x^*_j \notag \\
		& = \left(\psi \otimes \id_{\mathcal{O}_n}\right) \left(\sum_{i,j=1}^k a_i a^*_j \otimes x_i x^*_j \right) \notag \\
		& = \left(\psi \otimes \id_{\mathcal{O}_n}\right) \left( \left(\sum_{i=1}^k a_i \otimes x_i\right)\left(\sum_{i=1}^k a_i \otimes x_i\right)^*\right) \notag \\
		& \approx_{\delta/3} \left(\psi \otimes \id_{\mathcal{O}_n}\right)(uu^*) \notag \\
		& = 1_{F \otimes \mathcal{O}_n}.
	\end{align}
	Similarly, one can show $\left(\psi \otimes \id_{\mathcal{O}_n}\right)(u)^* \left(\psi \otimes \id_{\mathcal{O}_n}\right)(u) \approx_\delta 1_{F \otimes \mathcal{O}_n}$.
	Since we can assume $\delta$ is sufficiently small, there exists a unitary $v \in F\otimes \mathcal{O}_n$ such that $\|\left(\psi \otimes \id_{\mathcal{O}_n}\right)(u) - v\|< \epsilon/2$. It follows that
	\begin{align}
		\| \left(\theta\otimes\id_{\mathcal{O}_n}\right)(u) - \left(\eta\otimes\id_{\mathcal{O}_n}\right)(v)\| < \epsilon.
	\end{align}
	This yields that $\left(\theta\otimes\id_{\mathcal{O}_n}\right)(u)$ is homotopic to $\left(\eta\otimes\id_{\mathcal{O}_n}\right)(v)$. 
	
	On the other hand, the unitary group of any matrix amplification of $\mathcal{O}_n$ is connected (\cite[Theorem 1.9]{Cuntz}), so the unitary group  of $F\otimes \mathcal{O}_n$ is also connected. Moreover, $K_1(F\otimes\mathcal{O}_n)=0$, so $\left(\eta\otimes\id_{\mathcal{O}_n}\right)(v)$ is homotopic to $\left(\eta\otimes\id_{\mathcal{O}_n}\right)(1_{F\otimes \mathcal{O}_n}) = 1_{B \otimes \mathcal{O}_n}$.
	By transitivity, we obtain that $u$ is homotopic to $1_{B\otimes \mathcal{O}_n}$ and this implies $[\left(\theta\otimes\id_{\mathcal{O}_n}\right)(u)]_1=0$. As in the proof of Proposition \ref{K1map}, a similar argument holds for matrix amplifications, which in turn implies $K_1(\theta; \mathbb{Z}/n)= 0$.
     \end{proof}
     
     \begin{prop}\label{K0coeffmap}
     Let $A$ and $B$ be separable $\C$-algebras and $\theta:A\to B$ a $^*$-homomorphism with $K_1(\theta)=0$. Then the image of $K_0(\theta; \mathbb{Z}/n)$ is contained into $K_0(B)\otimes \mathbb{Z}/n$ for all $n\in\mathbb{N}$. In particular, this holds if $\theta$ is unital and $\dimnuc\theta=0$.
     \end{prop}

\begin{proof}
At the level of $K_0$-groups with coefficients, by  \eqref{Kthcoeffses}, we obtain the following short exact sequences with the maps induced by the morphism $\theta$ making each square commute
\[
\begin{adjustbox}{max width=\textwidth}
\begin{tikzcd}
0 \ar{r} & K_0(A)\otimes \mathbb{Z}/n \ar{r}\ar{d}{K_0(\theta)\otimes \id} & K_0(A;\mathbb{Z}/n) \ar{r}\ar{d}{K_0(\theta;\mathbb{Z}/n)} & \Tor(K_1(A), \mathbb{Z}/n) \ar{r}\ar{d}{K_1(\theta)} & 0 \\
0 \ar{r} & K_0(B)\otimes \mathbb{Z}/n \ar{r} & K_0(B;\mathbb{Z}/n) \ar{r}{q_B} & \Tor(K_1(B), \mathbb{Z}/n) \ar{r} & 0.
\end{tikzcd}
\end{adjustbox}
\]
Since $K_1(\theta)=0$, we obtain $q_B\circ K_0(\theta;\mathbb{Z}/n)=0$. This means precisely that the image of $K_0(\theta;\mathbb{Z}/n)$ is contained into $K_0(B)\otimes \mathbb{Z}/n$. The last part follows from Proposition \ref{K1map}.
\end{proof}

By combining \eqref{algK1decomp} and {\cite[Theorem 7.2]{rr0char}}, we see that a simple infinite dimensional AF-algebra has trivial Hausdorffised algebraic $K_1$-group. Therefore, any map which approximately factors through a simple AF-algebra is trivial at the level of the Hausdorffised algebraic $K_1$. Hence, to obtain our main results, the next proposition is a necessary condition.

     \begin{prop}\label{algK1map}
     Let $A$ be a simple, separable, unital, non-elementary $\C$-algebra and let $\theta:A\to B$ be a unital $^*$-homomorphism with $\dimnuc\theta=0$. Then $\overline{K_1}^{\alg}(\theta)=0$.

     \end{prop}

     \begin{proof}
     Since $\dimnuc\theta=0$ and $\theta$ is unital, Lemma \ref{homsfactoring} produces finite dimensional $\C$-algebras $(F_k)_{k\in \mathbb{N}}$ and unital $^*$-homomorphisms $\psi:A\to \prod\limits_{\infty}F_k$, $\eta:\prod\limits_{\infty}F_k \to B_{\infty}$ such that $\iota_B\circ\theta=\eta\circ\psi$. 
     We will prove this proposition by showing that $\overline{K_1}^{\alg}\Big(\prod\limits_{\infty}F_k\Big)=0$. Crucially, by simplicity of $A$, Corollary \ref{largematrixblocks} implies that, by passing to a subsequence, we can assume that the matrix blocks of $F_k$ have size at least $k$. This is the reason that underpins the fact that the Hausdorffised algebraic $K_1$-group of $\prod\limits_{\infty}F_k$ is trivial.

     It will then follow that $\overline{K_1}^{\alg}(\iota_B)\circ\overline{K_1}^{\alg}(\theta)=\overline{K_1}^{\alg}(\iota_B\circ\theta)=0$. 
     Using that $\overline{K_1}^{\alg}(\iota_B)$ is injective, since 
     $\iota_B$ induces an injective morphism at the level of the invariants ({\cite[Lemma 3.16]{classif}}), we obtain $\overline{K_1}^{\alg}(\theta)=0$.
     
     Now, let us prove the claim.  Observe that $K_1\Big(\prod\limits_{\infty}F_k\Big)=0$; and notice that by equation \eqref{algK1decomp}, the claim will follow after checking that the image of the pairing map is uniformly dense in $\Aff \Big(T\Big(\prod\limits_{\infty} F_k\Big)\Big)$. 
     
     In order to do this, we will follow the strategy employed in the proof of \cite[Theorem 7.2]{rr0char}. 
     For convenience, let us denote $\prod\limits_{\infty}F_k$ by $E$. Observe that $\overline{\rho_E(K_0(E))}$ is closed and separates points since $E$ has real rank zero. Moreover, it contains constant functions as $E$ is unital.
     
     After checking that $\overline{\rho_E(K_0(E))}$ is a subspace of $\Aff(T(E))$, we will reach the conclusion by applying a version of Kadison's Representation Theorem given in {\cite[Proposition 3.12 (b)]{BKR92}}. 
     Let us pick a non-zero positive element $[x]_0$ in $K_0(E)$. Since a matrix amplification of a sequence algebra is the sequence algebra of matrix amplifications, $x$ is induced by a sequence of projections $p_k\in \mathbb{M}_N(F_k)$ for some $N\in\mathbb{N}$.
     Then, if $F_k$ has a decomposition into $r$ matrix blocks with sizes $n_1^{(k)},\ldots, n_r^{(k)}$, we can write $p_k$ as a sum of orthogonal projections $p_k^{(1)}\oplus\ldots\oplus p_k^{(r)}$.\footnote{Note that $r$ depends on $k$, but we supress this for ease of notation.}
 
	 Let $n \in \mathbb{N}$ 
     and let us denote by $\mathrm{tr}_{Nn_i^{(k)}}$ the unique normalised trace on $\M_{Nn_i^{(k)}}$. Thus $$
     \mathrm{tr}_{Nn_i^{(k)}}(p_k^{(i)})=%
     \frac{t_i}{Nn_i^{(k)}}
     $$ 
     for some $t_i\in\mathbb{N}$. 
     If $\lfloor t_i / n \rfloor < n$, we set $q_k^{(i)}:=p_k^{(i)}$; 
     and if $\lfloor t_i / n \rfloor \geq n$, we set $q_k^{(i)} \in \M_{Nn_i^{(k)}}$ as the diagonal matrix with the first $\lfloor t_i/n\rfloor$ entries equal to $1$ and the rest equal to $0$. 
     Then, consider $y_n:=(q_k)_{k\in \mathbb{N}}$, where $q_k:=q_k^{(1)}\oplus \ldots \oplus q_k^{(r)}$. 
     Since the sizes of the blocks go to infinity and $n$ is fixed, note that the projections $p_k^{(i)}$ will be zero on any limit trace if $\lfloor t_i / n \rfloor < n$. Thus, by construction, we obtain that 
     \begin{align}
     	n\tau(y_n)\leq \tau(x)\leq (n+1)\tau(y_n)
     \end{align} 
     for all limit traces $\tau$ on $E$.

     Finally, the convex hull of limit traces on $E$ is weak$^*$-dense in $T(E)$ (by essentially the same proof of {\cite[Theorem 1.2]{nonsillytraces}}), which yields that for all $n\in\mathbb{N}$ there exists some non-zero positive $[y_n]_0\in K_0(E)$ such that 
     \begin{align}\label{eq: WeakDiv}
     	n\tau(y_n)\leq \tau(x)\leq (n+1)\tau(y_n)
     \end{align} 
     for all traces $\tau$ on $E$. 
     
     We claim that $\overline{\rho_E(K_0(E))}$ is a subspace of $\Aff(T(E))$. Note that by linearity and density of rationals in the reals, it suffices to show that for any $[x]_0\in K_0(E)_+$ and any $k\in\N$, $\frac{1}{k}\rho_E([x]_0)\in \overline{\rho_E(K_0(E))}$. Let $\epsilon>0, k\in\mathbb{N}$, and $n\in\mathbb{N}$ be such that 
     \begin{align}\label{eq:choice.k.n}
     	\frac{1}{nk}\tau(x)<\epsilon
     \end{align}
     for any $\tau\in T(E).$ Then, by \eqref{eq: WeakDiv}, there exists $[y]_0\in K_0(E)_+$ such that 
     \begin{align}\label{eq: subspaceineq}
     		nk \tau(y)\leq \tau(x)\leq (nk+1)\tau(y)
     \end{align} for all $\tau\in T(E)$. 
 	By rearranging the inequality in \eqref{eq: subspaceineq} and dividing by $k$, we obtain 
 	\begin{align}
 		0\leq \frac{1}{k}\tau(x)-n\tau(y)\leq \frac{1}{k}\tau(y)
 	\end{align}
 	for any $\tau\in T(E)$.
     Moreover, by  the left-hand side of \eqref{eq: subspaceineq},  we have
     \begin{align}\label{eq:lhs}
     	\tau(y)\leq \frac{1}{nk}\tau(x).
     \end{align}
 	Thus, we obtain
 	\begin{align}
 			0\leq \frac{1}{k}\tau(x)-n\tau(y)\leq \frac{1}{k}\tau(y) \overset{\eqref{eq:lhs}}{<} \frac{1}{n k^2}\tau(x) \overset{\eqref{eq:choice.k.n}}{<} \epsilon
 	\end{align}
     for all $\tau\in T(E)$. Hence, we have shown that for all $[x]_0 \in K_0(E)_+$ and $\epsilon>0$ there exist $n\in\mathbb{N}$ and $[y]_0\in K_0(E)_+$ such that 
     $$0\leq \frac{1}{k}\rho_E([x]_0) - n \rho_E([y]_0)<\epsilon.$$ 
     Since $\epsilon$ was arbitrary, this yields that $\frac{1}{k}\rho_E([x]_0)\in \overline{\rho_E(K_0(E))}$. Therefore
      $\overline{\rho_E(K_0(E))}$ is a subspace of $\Aff(T(E))$ and we reach the conclusion by {\cite[Proposition 3.12 (b)]{BKR92}}. 
     \end{proof}

For a $^*$-homomorphism with nuclear dimension equal to zero, we would like to have a better understanding of its associated $K_0$-map. In particular, we note that Lemma \ref{homupwardmap} essentially gives finite dimensional subalgebras of the codomain; which in turn implies the existence of projections. In order to give a more precise description, we will need the following notion of weak divisibility.

\begin{defn}[{\cite[Definition 5.1]{AFembs}}]\label{weakdiv}
Let $A$ be a $\C$-algebra and $p$ be a non-zero projection in $A$. It is said that $A$ is \emph{weakly divisible of degree $n$} at $p$ if there is a unital $^*$-homomorphism $$\phi: \M_{n_1}\oplus \M_{n_2}\oplus \ldots \oplus \M_{n_r}\to pAp,$$ for some natural numbers $r,$ and $n_1, n_2, \ldots, n_r$ where $n_j\geq n$ for all $j$.
\end{defn}
	
This condition is equivalent to a notion of divisibility in the Murray-von Neumann semigroup $V(A)$ of $A$. Precisely, $A$ is weakly divisible of degree $n$ at $p$ if and only if there are natural numbers $r, n_1,\ldots, n_r$ and $x_1,\ldots, x_r \in V(A)$ such that $n_j\geq n$ for all $j$ and $[p]_0=n_1x_1+\ldots+n_rx_r$.

\begin{cor}\label{codomainweaklydiv}
	Let $A$ and $B$ be unital separable $\C$-algebras with $A$ simple and non-elementary.
	If $\theta:A\to B$ is a unital $^*$-homomorphism with $\dimnuc\theta=0$, then $B$ is weakly divisible of degree $n$ at $1_B$ for all $n\in \mathbb{N}$.
\end{cor}
\begin{proof}
	This follows from Lemma \ref{homupwardmap} and Corollary \ref{largematrixblocks}.
\end{proof}
	
	\begin{rmk}\label{remark:explicit.weak.div}
		In fact, one can obtain a more explicit condition. Let $n\in \mathbb{N}$ and suppose that $B$ is weakly divisible of degree $m$ at $1_B$ for all $m\in\mathbb{N}$. Note that for all $m\geq n(n+2)$, there exist $r,s\in \mathbb{N}$ such that $m=rn+s(n+1)=n(r+s)+s$. Indeed,  if $m$ is $0$ modulo $n$, let $s=n$ and since $m\geq n(n+2)$, we can take $r\geq 1$ such that $m=n(r+n+1)$. Else, let $s$ be the non-zero remainder when we divide $m$ by $n$ and this automatically determines $r$. Moreover, if $B$ has cancellation of projections (for example if $B$ has stable rank one), then for all $n\in \mathbb{N}$, there exist $g^{(n)}, h^{(n)}\in K_0(B)_+$ such that $[1_B]_0=ng^{(n)}+(n+1)h^{(n)}$. Then, there exists a unital $^*$-homomorphism from $\M_n \oplus \M_{n+1}$ to $B$ by {\cite[Lemma 1.3.1]{rordambook}}.
	\end{rmk}

	If $B$ is classifiable, the condition described in Remark \ref{remark:explicit.weak.div} is connected with the property of having real rank zero. 
	Any simple, separable, unital, finite, $\mathcal{Z}$-stable $\C$-algebra with real rank zero is weakly divisible of degree $n$ at $1$ for all $n\in \mathbb{N}$ ({\cite[Theorem 7.2]{rr0char}}). 
	The converse holds in the unique trace case (\cite[Corollary 7.3]{rr0char}), but it is not true in general. For instance, one can consider a classifiable $\C$-algebra $A$ with $K_0(A)=\mathbb{Q}$, $K_1(A)=0$, $[1_A]_0=1$, and $2$ extreme traces. Since there is only one state on $K_0(A)$, projections do not separate traces, so $A$ cannot have real rank zero.

 The next proposition is the key technical step in the proof of Theorem \ref{factoringintoBinfty}. Roughly speaking, if we have a $^*$-homomorphism factoring through a sequence algebra of finite dimensional $\C$-algebras, under enough regularity assumptions, we can show that the respective $^*$-homomorphism has the behaviour of a map which factors through a simple AF-algebra at the level of total $K$-theory and traces. This proposition shows precisely how to obtain such an AF-algebra by finding a suitable dimension group inside the $K_0$-group of the given sequence algebra.

\begin{prop}\label{prop:AFalgebra.compatibility}
	Let $A$ be a simple separable unital exact non-ele-mentary and stably finite $\C$-algebra and let $(F_k)_{k \in \mathbb{N}}$ be a sequence of finite dimensional $\C$-algebras.
	Suppose $\psi: A \to \prod\limits_{\infty} F_k$ is a unital $^*$-homomorphism.
	Then there exist a simple separable unital AF-algebra $C$ with no finite dimensional representations, together with $\Lambda$-morphisms 
	$$\underline{\alpha}: \underline{K}(A) \to \underline{K}(C), \qquad \underline{\beta}: \underline{K}(C) \to \underline{K}\left( \prod\limits_{\infty} F_k \right)$$
	and affine maps 
	$$\gamma_0:  \Aff(T(A)) \to  \Aff(T(C)), \qquad \gamma_1:  \Aff(T(C)) \to  \Aff\left(T\Big(\prod\limits_{\infty} F_k\Big)\right)$$
	such that 
	\begin{align}
		\underline{K}(\psi) = \underline{\beta} \circ \underline{\alpha}, \qquad \text{and} \qquad \Aff(T(\psi)) = \gamma_1 \circ \gamma_0,
	\end{align}
	and the following diagram commutes
	\begin{equation}\label{dimgroup.diagram}
		\begin{tikzcd}
			K_0(A) \ar{r}{\alpha_0}\ar{d}{\rho_A} & K_0(C) \ar{r}{\beta_0}\ar{d}{\rho_C} & K_0\Big(\prod\limits_{\infty} F_k\Big) \ar{d}{\rho_{F_\infty}} \\
			\Aff(T(A)) \ar{r}{\gamma_0} & \Aff(T(C)) \ar{r}{\gamma_1} & \Aff\Big(T\Big(\prod\limits_{\infty} F_k\Big)\Big).
		\end{tikzcd}
	\end{equation}
\end{prop}

\begin{proof}
By Proposition \ref{dimgroup}, there exists a unital separable $\C$-algebra $E$ such that $\psi(A)\subset E\subset \prod\limits_{\infty} F_k$ and $K_0(E)$ is a dimension group with Property (D). 
Let $K_0(E)_{++}$ denote the set of order units of $K_0(E)$.
Since $E$ contains the unit, it follows that $K_0(E)_{++}$ is non-empty. Then $(K_0(E), K_0(E)_{++}\cup \{0\})$ is a simple dimension group by \cite[Proposition 10]{embZ}. By Proposition \ref{AFdimgroup}, there is a simple unital AF-algebra $C$ with $(K_0(C),K_0(C)_+)=(K_0(E), K_0(E)_{++}\cup \{0\})$. We will freely identify these groups. Moreover, since $K_0(C)$ has property (D), $C$ has no finite dimensional representations by {\cite[Lemma 9]{embZ}}. To define suitable $K$-theory maps we will essentially use that, apart from the order in $K_0$, the $\C$-algebras $E$ and $C$ have the same $K$-groups and $K$-groups with coefficients.

If $\psi_0:A\to E$ is the corestriction of $\psi$, define $\alpha_0:K_0(A)\to K_0(C)$ by $\alpha_0:=K_0(\psi_0)$. As $K_0(A)$ is a simple ordered group (\cite[Corollary 6.3.6]{blackadar}), all non-zero positive elements in $K_0(A)$ are order units, so $\alpha_0$ is a well-defined ordered group homomorphism. Since all the $K_1$-groups with coefficients of $C$ vanish, we define $\alpha_1: K_1(A)\to K_1(C)$ and $\alpha_1^{(n)}:K_1(A;\mathbb{Z}/n)\to K_1(C;\mathbb{Z}/n)$ to be the corresponding zero map. It remains to define $\alpha_0^{(n)}:K_0(A;\mathbb{Z}/n)\to K_0(C;\mathbb{Z}/n)$. For this, note that $K_0(C;\mathbb{Z}/n)=K_0(C)\otimes \mathbb{Z}/n=K_0(E)\otimes \mathbb{Z}/n$. Moreover, $K_1(E)=0$ by Proposition \ref{dimgroup}, so $K_0(E;\mathbb{Z}/n)=K_0(E)\otimes \mathbb{Z}/n$ by equation \eqref{Kthcoeffses}, which in turn implies that $K_0(C;\mathbb{Z}/n)=K_0(E;\mathbb{Z}/n)$. Thus, we can set $\alpha_0^{(n)}:=K_0(\psi_0;\mathbb{Z}/n)$. Note that the absence of torsion in the $K_0$-groups yields that $K_1(C;\mathbb{Z}/n)=K_1(E;\mathbb{Z}/n)$ for all $n\geq 2$. Finally, $\underline{\alpha}$ is a well-defined $\Lambda$-morphism since $\underline{K}(\psi_0)$ is a $\Lambda$-morphism.

There is an inclusion $\iota:E\to \prod\limits_{\infty} F_k$ such that $\iota\circ\psi_0=\psi$. Define $\beta_1:=0$ and $\beta_1^{(n)}:=0$. Similarly, identifying $K_0(C)$ with $K_0(E)$ and $K_0(C;\mathbb{Z}/n)$ with $K_0(E;\mathbb{Z}/n)$, let us define $\beta_0:=K_0(\iota)$ and $\beta_0^{(n)}:=K_0(\iota;\mathbb{Z}/n)$. Likewise, $\underline{\beta}$ is a $\Lambda$-morphism because $\underline{K}(\iota)$ is a $\Lambda$-morphism. Furthermore, by construction and because composition of $\Lambda$-morphisms is a $\Lambda$-morphism, it follows that $\underline{\beta}\circ\underline{\alpha}=\underline{K}(\psi)$.

For the second part, by Proposition \ref{AFdimgroup} and 
\cite[Proposition 11]{embZ}, it follows that the space of states on $(K_0(E), K_0(E)_+)$ is affinely homeomorphic to the space of states on $(K_0(C),K_0(C)_+)$. Since $E$ is stably finite, unital, and with real rank zero, the space of states on $K_0(E)$ is affinely homeomorphic to $QT(E)$, the space of quasitraces on $E$ by {\cite[Theorem 6.9.1]{blackadar}}. As $A$ is exact, quasitraces on $A$ are traces \cite{haagerup}. Moreover, since $C$ is an AF-algebra,  states on $K_0(C)$ are traces on $C$. With these identifications, consider $\tilde{\gamma}_0$ to be given by the following sequence of maps 
\[
\begin{tikzcd}
\tilde{\gamma}_0: T(C)\ar{r}{\cong} & S(K_0(E)) \ar{r}{\cong} & QT(E) \ar{r}{(\psi_0)_*} & QT(A) = T(A).
\end{tikzcd}
\] Let $\gamma_0$ be the dual map of $\tilde{\gamma}_0$.

Set $\tilde{\gamma}_1$ to be the following sequence of maps
 \[
\begin{tikzcd}
\tilde{\gamma}_1: T\Big(\prod\limits_{\infty}F_k\Big) \ar{r}{i} & QT\Big(\prod\limits_{\infty}F_k\Big) \ar{r}{\sigma_0} & QT(E) \ar{r}{\cong} & T(C),
\end{tikzcd}
\]
where $i$ is the canonical inclusion and $\sigma_0$ is the affine map given by restriction. Likewise, let $\gamma_1$ be the dual map of $\tilde{\gamma}_1$. By construction, it follows that $\Aff(T(\psi)) = \gamma_1\circ\gamma_0$. 

Finally, note that the diagram in \eqref{dimgroup.diagram} commutes by construction. Precisely, $\alpha_0=K_0(\psi_0)$, $T(C)\cong QT(E)$, and $QT(A)=T(A)$ give that the left square commute. Similarly, by using the canonical map $\Aff\Big(QT\Big(\prod\limits_{\infty} F_k\Big)\Big)\to \Aff\Big(T\Big(\prod\limits_{\infty} F_k\Big)\Big)$, the right square commutes as $\beta_0=K_0(\iota)$ and $T(C)\cong QT(E)$. 
\end{proof}

\section{Factoring through simple AF-algebras}\label{mainresults}

We now proceed to prove the first main theorem which states that a classifiable zero dimensional $^*$-homomorphism factors through a simple AF-algebra when viewed as a map into the sequence algebra of its codomain.

\begin{proof}[Proof of Theorem \ref{factoringintoBinfty}]

If $A$ is finite dimensional then the statement is vacuously true, so let us suppose $A$ is infinite dimensional. 
Observe then that simplicity and unitality imply that $A$ has no finite dimensional representations. 
First, note that if we assume that $\iota_B \circ \theta$ factors through a simple separable AF-algebra , then $\iota_B \circ \theta$ has nuclear dimension equal to $0$ and hence $\theta$ also has nuclear dimension equal to $0$ by {\cite[Proposition $2.5$]{decomprankZstable}}. 

Conversely, let us assume $\theta$ has nuclear dimension equal to zero. By Lemma \ref{homsfactoring}, there are finite dimensional $\C$-algebras $(F_k)_{k \in \mathbb{N}}$ and unital $^*$-homomorphisms $\sigma:A\to\prod\limits_\infty F_k$, $\eta:\prod\limits_\infty F_k\to B_\infty$ such that $\iota_B \circ \theta = \eta \circ \sigma$.
By Proposition \ref{prop:AFalgebra.compatibility}, there exist a simple unital AF-algebra $C$ with no finite dimensional representations, $\Lambda$-morphisms 
$$\underline{\alpha}: \underline{K}(A) \to \underline{K}(C), \qquad \underline{\beta}: \underline{K}(C) \to \underline{K}\left( \prod\limits_{\infty} F_k \right)$$
and affine maps 
$$\gamma_0:  \Aff(T(A)) \to  \Aff(T(C)), \qquad \gamma_1:  \Aff(T(C)) \to  \Aff\left(T\Big(\prod\limits_{\infty} F_k\Big)\right)$$
such that 
\begin{align}\label{mainthm1:eq1}
	\underline{K}(\sigma) = \underline{\beta} \circ \underline{\alpha}, \qquad \text{and} \qquad \Aff(T(\sigma)) = \gamma_1 \circ \gamma_0,
\end{align}
and the diagram \eqref{dimgroup.diagram} commutes. 
Additionally, since $C$ is a simple unital infinite dimensional AF-algebra then 
$\overline{K_1}^\alg (C) =0$ by combining \cite[Theorem 7.2]{rr0char} and \eqref{algK1decomp}. 

Since the diagram \eqref{dimgroup.diagram} commutes and $\overline{K_1}^\alg (C) =0$, then the triple $(\underline{\alpha}, 0, \gamma_0)$ makes the diagram \eqref{KthTracepairing} commutative.
By Lemma \ref{lemma : classifthmrr0}, it follows that $(\underline{\alpha}, 0, \gamma_0): \underline{K}T_u(A) \to \underline{K}T_u(C)$ is a $\underline{K}T_u$-morphism, which is automatically faithful by the simplicity of $A$. Then, by Theorem \ref{classifthm} there exists a unital $^*$-homomorphism $\psi: A \to C$ realising $(\underline{\alpha}, 0, \gamma_0)$.
On the other hand, $(\underline{\beta}, 0, \gamma_1): \underline{K}T_u(C) \to \underline{K}T_u\Big( \prod\limits_{\infty} F_k \Big)$ is a $\underline{K}T_u$-morphism. Then, $\underline{K}T_u(\eta)%
\circ (\underline{\beta}, 0, \gamma_1): \underline{K}T_u(C) \to \underline{K}T_u(B_\infty)$ induces a unital $^*$-homomorphism $\varphi: C \to B_{\infty}$ by \cite[Theorem 9.1]{classif}. 

By construction and equation \eqref{mainthm1:eq1}, we obtain
\begin{align}
	\underline{K}(\varphi \circ \psi) &= \underline{K}(\varphi) \circ \underline{K}(\psi) = \underline{K}(\eta) \circ \underline{\beta} \circ \underline{K}(\psi)  \notag \\
	& = \underline{K}(\eta) \circ \underline{\beta} \circ \underline{\alpha} = \underline{K}(\eta) \circ \underline{K}(\sigma) = \underline{K}(\eta \circ \sigma) = \underline{K}(\iota_B \circ \theta).
\end{align}
Similarly,
\begin{align}
	\Aff(T(\varphi \circ \psi)) & = \Aff(T(\varphi)) \circ \Aff(T(\psi)) = \Aff(T(\eta)) \circ \gamma_1 \circ \gamma_0 \notag \\
	& = \Aff(T(\eta)) \circ \Aff(T(\sigma)) = \Aff(T(\iota_B \circ \theta)).
\end{align}
Using again that $\overline{K_1}^\alg (C) =0$, we obtain that $\overline{K_1}^\alg (\varphi \circ \psi) = 0$. On the other hand, Proposition \ref{K1map} yields that $\overline{K_1}^\alg(\iota_B \circ \theta) =0$. In summary,
\begin{align}
\underline{K}T_u(\varphi \circ \psi) = \underline{K}T_u(\iota_B \circ \theta). \notag 
\end{align}

Therefore, $\iota_B\circ\theta$ and $\varphi\circ\psi$ are unitarily equivalent by {\cite[Theorem $9.1$]{classif}}. By conjugating $\varphi$ by a unitary if necessary, we can assume that $\iota_B\circ\theta=\varphi\circ\psi$ and the conclusion follows.
\end{proof}

In general, the AF-algebra $C$ identified above can be embedded into $B_\infty$ but not into $B$. It is not clear how to find a dimension group which is factoring the $K_0$-map in general. 
However, this is not an issue if we assume that the domain or codomain has real rank zero. 
Let us proceed to prove Theorem \ref{rr0codomain}, where the codomain has real rank zero, and then we will show that a similar result holds if the domain has real rank zero instead. Let us first record a result which is well-known to experts. 

\begin{lemma}\label{lem: AFskeletonRR0}
   Let $B$ be a simple, separable, unital, nuclear, $\Z$-stable $\C$-algebra which has real rank zero, $T(B) \neq \emptyset$ and $K_0(B)$ is torsion free. Then there exists a simple, unital, infinite dimensional AF-algebra $C$ such that $K_0(C)\cong K_0(B)$ as ordered abelian groups and $T(C)\cong T(B)$.
\end{lemma}

\begin{proof}
We first claim that $K_0(B)$ is a simple dimension group. Note that $K_0(B)$ is a simple ordered abelian group (\cite[Corollary 6.3.6]{blackadar}). 
Suppose $n[p]_0\geq 0$ for some $n\in\mathbb{N}$ and $[p]_0\in K_0(B)_+$. Since $K_0(B)$ is a simple ordered abelian group, either $n[p]_0=0$ or $n[p]_0$ is an order unit. 
In the former case, $[p]_0=0$ since $K_0(B)$ is torsion free, and in the latter, by compactness of $T(B)$ we see that there is some $\alpha>0$ such that $\tau(p^{\oplus n})>\alpha$ for all $\tau\in T(B)$. 
Therefore, $\tau(p)>0$ for all $\tau\in T(B)$, which gives that $p\geq 0$ since $B$ has strict comparison of projections with respect to tracial states (\cite[Corollary 4.6]{rr0char}). Thus, $K_0(B)$ is unperforated.

Finally, note that since $B$ has real rank zero, the projections in $B$ satisfy the Riesz decomposition with respect to the Murray-von Neumann equivalence (\cite[Corollary 1.3]{riesz}). As $B$ has stable rank one, $B$ has cancellation (\cite[Corollary 6.5.1]{blackadar}) and hence $K_0(B)$ satisfies the Riesz decomposition property. Since the Riesz decomposition property is equivalent to the Riesz interpolation property by {\cite[Proposition 2.1]{goodearlbook}}, it follows that $K_0(B)$ is a dimension group ({\cite[Theorem 1.4.4]{rordambook}}).

Then, by Proposition \ref{AFdimgroup}, there is a simple, unital infinite dimensional AF-algebra $C$ with 
$$(K_0(C), K_0(C)_+, [1_C]_0)\cong (K_0(B), K_0(B)_+, [1_B]_0).$$  As $B$ is unital, nuclear with real rank zero and stable rank one, traces on $B$ are homeomorphic to states on $K_0(B)$ ({\cite[Theorem 6.9.1]{blackadar}}). 
Likewise, traces on $C$ are homeomorphic to states on $K_0(C)$, so $T(B)\cong T(C)$.
\end{proof}

\begin{proof}[Proof of Theorem \ref{rr0codomain}]

If $\theta$ is approximately unitarily equivalent to a $^*$-homomorphism which factors through a simple AF-algebra, then $\theta$ has nuclear dimension equal to zero. On the other hand, if $\theta$ has nuclear dimension equal to zero, then $K_1(\theta)=0$ and $K_1(\theta; \mathbb{Z}/n)=0$ for all $n\geq 2$ by Proposition \ref{K1map} and Proposition \ref{K1coeffmap}, respectively. This shows that (i)$\implies$(ii)$\implies$(iii).

We will prove that (iii)$\implies$(i). Let us suppose that the $K_1$-maps with coefficients induced by $\theta$ are all trivial. We claim that $\theta$ is approximately unitarily equivalent to a $^*$-homomorphism which factors through a simple AF-algebra. By Lemma \ref{lem: AFskeletonRR0}, there is a simple unital infinite dimensional AF-algebra $C$ with 
$$(K_0(C), K_0(C)_+, [1_C]_0)\cong (K_0(B), K_0(B)_+, [1_B]_0)$$ and $T(C)\cong T(B)$. In order to simplify our proof we will freely identify these two ordered dimension groups. We claim that, up to approximate unitary equivalence, $\theta$ factors through $C$.

In order to show this, we are going to use Theorem \ref{classifthm} %
to construct unital $^*$-homomorphisms from $A$ to $C$ and from $C$
to $B$, respectively. 
We start by obtaining a $\underline{K}T_u$-morphism between the $\underline{K}T_u$-invariants of $A$
and $C$. 
Let $\alpha_0:K_0(A)\to K_0(C)$ be given by $\alpha_0:=K_0(\theta)$. Since $K_1(C)=0$, we will let $\alpha_1:=0$. At the level of the $K$-groups with
coefficients, we recall the short exact sequence in \eqref{Kthcoeffses}
\[
\begin{adjustbox}{max width=\textwidth}
\begin{tikzcd}
0 \ar{r} & K_i(C)\otimes \mathbb{Z}/n \ar{r} & K_i(C;\mathbb{Z}/n) \ar{r} & \Tor(K_{1-i}(C), \mathbb{Z}/n) \ar{r} & 0.
\end{tikzcd}
\end{adjustbox}
\]
Since $K_0(C)$ is torsion free and $K_1(C)=0$, it follows that $K_1(C;\mathbb{Z}/n)=0$, which in turn implies that we must take $\alpha_1^{(n)}:=0$ for all $n \in \mathbb{N}$.

Using the same short exact sequence from above and that $K_1(C)=0$, it follows that $$K_0(C;\mathbb{Z}/n) \cong K_0(C)\otimes \mathbb{Z}/n\cong K_0(B)\otimes\mathbb{Z}/n.$$
By Proposition \ref{K0coeffmap}, the image of $K_0(\theta; \mathbb{Z}/n)$ is contained in $K_0(B) \otimes \mathbb{Z}/n$, and after identifying $K_0(C, \mathbb{Z}/n)$ with $K_0(B) \otimes \mathbb{Z}/n$, we can define
$\alpha_0^{(n)}:K_0(A;\mathbb{Z}/n)\to K_0(C;\mathbb{Z}/n)$ as $\alpha_0^{(n)}:=K_0(\theta;\mathbb{Z}/n)$.
We denote by $\underline{\alpha}$ the collection of all the morphisms $\alpha_i, \alpha_i^{(n)}$ defined above. Since these maps are either induced by $\theta$ or are the zero map, it is immediate that they all intertwine the Bockstein maps. Hence, $\underline{\alpha}$ will be a compatible $\Lambda$-morphism at the level of total K-theory. 

Recall from \eqref{algK1decomp} that $$\overline{K_1}^{\alg}(C)\cong \Aff(T(C))/\overline{\im\rho_C}\oplus K_1(C).$$ Since $C$ is a simple unital infinite dimensional AF-algebra, it follows that $\overline{K_1}^{\alg}(C)=0$ (again by combining \eqref{algK1decomp} and {\cite[Theorem 7.2]{rr0char}}). 
Thus, the only group homomorphism from $\overline{K_1}^{\alg}(A)$ to $\overline{K_1}^{\alg}(C)$ is the zero morphism.

Recall that $T(C) \cong S(K_0(C)) \cong S(K_0(B))\cong T(B)$ (again, we will freely identify these simplices). Hence, we can define $\gamma:\Aff(T(A)) \to \Aff(T(C))$ as $\gamma:=\Aff(T(\theta))$.
It follows that the diagram \eqref{KthTracepairing} commutes, and by Lemma \ref{lemma : classifthmrr0}, 
$(\underline{\alpha},0,\gamma)$ defines a $\underline{K}T_u$-morphism from $\underline{K}T_u(A)$ to $\underline{K}T_u(C)$. Thus, there is a unital $^*$-homomorphism $\psi:A\to C$ with associated $\underline{K}T_u$-invariant equal to $(\underline{\alpha},0,\gamma)$ by Theorem \ref{classifthm}.  
Set $\beta_0:=\id_{K_0(B)}$, $\beta_1:=0$, $\beta_1^{(n)}:=0$, and $\beta_0^{(n)}:=\id_{K_0(B)\otimes\mathbb{Z}/n}$ for all $n\geq 2$. By Lemma \ref{lemma : classifthmrr0}, $(\underline{\beta}, 0, \id_{\Aff(T(B))}): \underline{K}T_u(C) \to \underline{K}T_u(B)$ gives rise to a
unital $^*$-homomorphism $\varphi:C\to B$. 
We will check now that $\theta$ and $\varphi\circ\psi$ are indeed approximately unitarily equivalent. It is clear by construction that 
\begin{align*}
\Aff(T(\varphi\circ\psi))&=\Aff(T(\varphi)) \circ \Aff(T(\theta))\\ &= \id_{\Aff(T(B))} \circ \Aff(T(\theta))\\ &=\Aff(T(\theta)).
\end{align*}
Similarly
$$K_0(\varphi\circ\psi)=K_0(\varphi) \circ K_0(\psi) = \id_{K_0(B)} \circ K_0(\theta) = K_0(\theta).$$
At the level of $K_1$-groups (with coefficients), since $K_1(C;\mathbb{Z}/n)=0$ and $K_1(C)=0$ it follows that $K_1(\varphi\circ\psi;\mathbb{Z}/n)=0$ and  $K_1(\varphi \circ \psi)=0$, respectively.
Moreover, by assumption $K_1(\theta;\mathbb{Z}/n)=0$ and $K_1(\theta)=0$.

We have then shown that $\underline{K}(\theta) = \underline{K}(\varphi \circ \psi)$ and $\Aff(T(\theta)) = \Aff(T(\varphi \circ \psi))$. By Theorem \ref{classifthm}, $\theta \sim_{\au} \varphi \circ \psi$ and the conclusion follows.
\end{proof}

A similar characterisation can also be obtained when the domain is classifiable, stably finite, and with real rank zero.

\begin{theorem}\label{rr0domain}
Let $A$ and $B$ be a simple, separable, unital, nuclear $\Z$-stable $\C$-algebras such that $A$ is stably finite, has real rank zero, satisfies the UCT, and $K_*(A)$ is torsion free.
Let $\theta:A\to B$ a unital $^*$-homomorphism. Then the following are equivalent:

\begin{enumerate}
    \item There exist a simple separable unital AF-algebra $C$ and $^*$-homomorphisms $\psi: A \to C$ and $\varphi: C \to B$ such that $\theta$ is approximately unitarily equivalent to $\varphi \circ \psi$;\label{factoring21}
    \item $\theta$ has nuclear dimension equal to zero;\label{factoring22}
    \item $K_1(\theta; \mathbb{Z}/n)=0$ for all $n\in\mathbb{N}$ and $\overline{K_1}^{\alg}(\theta)=0$.\label{factoring23}
    \end{enumerate}
\end{theorem}

\begin{proof}
The fact that (i) implies (ii) is clear and by combining Propositions \ref{K1coeffmap} and \ref{algK1map}, we obtain that (ii) implies (iii). It only remains to prove that (iii) implies (i). Here is where the extra real rank zero hypothesis enters into play.

Suppose that $K_1(\theta; \mathbb{Z}/n)=0$ for all $n\in\mathbb{N}$ and $\overline{K_1}^{\alg}(\theta)=0$. We claim that $\theta$ is approximately unitarily equivalent to a $^*$-homomorphism which factors through a simple unital AF-algebra. By Lemma \ref{lem: AFskeletonRR0}, there exists a simple, unital, infinite dimensional AF-algebra $C$ such that $K_0(C)\cong K_0(A)$ as ordered abelian groups and $T(A) \cong T(C)$ as Choquet simplices. We will freely identify these objects.

Our first aim is to define a $\Lambda$-morphism between $\underline{K}(A)$ and $\underline{K}(C)$.
We can take $\alpha_0 : K_0(A) \to K_0(C)$ as $\alpha_0 := \id_{K_0(C)}$ and $\alpha_1 := 0$. 
Since $K_1(A)$, and $K_1(C)$ are torsion free, equation \eqref{Kthcoeffses} yields that
\begin{align}\label{mainthm3:eq1}
	K_0(A; \mathbb{Z}/n) \cong K_0(A) \otimes \mathbb{Z}/n, \qquad K_0(C; \mathbb{Z}/n) \cong K_0(C) \otimes \mathbb{Z}/n.
\end{align}
Similarly, since $K_1(C) = 0$, equation \eqref{Kthcoeffses} yields $K_1(C;\mathbb{Z}/n)=0$. 
We can then take $\alpha_0^{(n)}: K_0(A; \mathbb{Z}/n) \to K_0(C; \mathbb{Z}/n)$ to be the identity map and $\alpha_1^{(n)}: K_1(A;\mathbb{Z}/n) \to K_1(C; \mathbb{Z}/n)$ as the zero map.

Essentially by definition, these maps intertwine the Bockstein maps and hence the collection of these maps defines a $\Lambda$-morphism $\underline{\alpha}: \underline{K}(A) \to \underline{K}(C)$.
By taking $\gamma_0 := \id_{\Aff(T(C))}$, it is immediate that the corresponding diagram \eqref{KthTracepairing} commutes. 
By Lemma \ref{lemma : classifthmrr0}, we obtain a $^*$-homomorphism $\psi: A \to C$ such that $\underline{K}(\psi) = \underline{\alpha}$ and $\Aff(T(\psi)) = \id_{\Aff(T(C))}$.

We now proceed to define a faithful morphism from the total invariant of $C$ to the total invariant of $B$. 
Since $K_0(C)\cong K_0(A)$ as ordered abelian groups and $K_1(C) =0$, we can set $\beta_i: K_i(C) \to K_i(B)$ as $\beta_0 := K_0(\theta)$ and $\beta_1 := 0$.
Recall we have already observed that $K_0(A; \mathbb{Z}/n) \cong K_0(C;\mathbb{Z}/n)$ and $K_1(C; \mathbb{Z}/n) = 0$. Then, we can define $\beta_i^{(n)}: K_i(C; \mathbb{Z}/n) \to K_i(B; \mathbb{Z})$ by 
\begin{align}
	\beta_0^{(n)}= K_0(\theta, \mathbb{Z}/n) \qquad \text{and} \qquad \beta_1^{(n)} = 0.
\end{align}
Note that $K_1(\theta)=0$ and $K_1(\theta;\mathbb{Z}/n)=0$ for all $n\geq 2$ by assumption. Therefore, the collection of maps $\underline{\beta}$ can be identified with the collection of maps $\underline{K}(\theta)$. Hence, these maps define a $\Lambda$-morphism $\underline{\beta}: \underline{K}(C) \to \underline{K}(B)$.

On the other hand, since $C$ is a simple unital infinite dimensional AF-algebra, it follows that $\overline{K_1}^\alg(C) = 0$ by \cite[Theorem 7.2]{rr0char}. Hence, we can take the map $\overline{K_1}^\alg(C) \to \overline{K_1}^\alg(B)$ to be the zero map.
Likewise, as $T(C)\cong T(A)$, 
we can set $\gamma_1:= \Aff (T(\theta)): \Aff(T(C))\to \Aff(T(B))$. It follows that the corresponding diagram \eqref{KthTracepairing} commutes. Moreover, the diagram \eqref{newpairingdiagram1} commutes automatically since $K_1(C)=0$ by \cite[Proposition 3.12]{classif}.
Then, Theorem \ref{classifthm} yields a $^*$-homomorphism $\varphi: C \to B$ such that $\underline{K}(\varphi) = \underline{\beta}$, $\overline{K_1}^\alg(\varphi)=0$ and $\Aff(T(\varphi)) = \Aff(T(\theta))$.

We will complete the proof by showing that
$\theta$ and $\varphi\circ\psi$ are approximately unitarily equivalent.
By hypothesis $K_1(\theta; \mathbb{Z}/n) =0$ and $\overline{K_1}^{\alg}(\theta)=0$. Thus, straightforward calculations yield
\begin{align}
	\underline{K}(\varphi \circ \psi) = \underline{\beta} \circ \underline{\alpha}= \underline{K}(\theta).
\end{align}
Similarly,
\begin{align}
	\Aff(T(\varphi\circ\psi)) = \Aff(T(\theta)) \circ \id_{\Aff(T(B))} = \Aff(T(\theta)).
\end{align}
Therefore, by Theorem \ref{classifthm}, we conclude that $\theta\sim_{\au}\varphi\circ\psi$.
\end{proof}

\begin{rmk}
Unlike in Theorem \ref{rr0codomain}, the proof above depends on the assumption that $K_1(A)$ is torsion free. The reason is that unless $K_1(A)$ is torsion free, $K_0(\theta;\mathbb{Z}/n)$ can potentially hit elements outside the image of $K_0(\theta)\otimes\id_{\mathbb{Z}/n}$. However, composing any map $K_0(A;\mathbb{Z}/n)\to K_0(C;\mathbb{Z}/n)$ with $K_0(C;\mathbb{Z}/n)=K_0(C)\otimes\mathbb{Z}/n\to K_0(B)\otimes\mathbb{Z}/n$ will only hit elements in the image of $K_0(\theta)\otimes\id_{\mathbb{Z}/n}$.
\end{rmk}

\section{Embeddings of $\Z$}\label{embZ}

In this last section we turn our attention to unital embeddings of $\Z$. 
By analysing the weak divisibility condition provided by Corollary \ref{codomainweaklydiv}, we will characterise embeddings of $\Z$ with nuclear dimension equal to zero.

The lemma below is the main technical result we need in order to build unital embeddings of $\mathcal{Z}$ which have nuclear dimension $0$. Essentially, one can witness quasidiagonality of $\mathcal{Z}$ via any sequence of matrix algebras with arbitrarily large sizes.
	
\begin{lemma}\label{QDofZ}
Let $\mathcal{F}\subset\mathcal{Z}$ be a finite set and $\epsilon>0$. Then there exists $N\in\mathbb{N}$ such that for all $m\geq N$ there is a ucp map $\phi_m:\mathcal{Z}\to \M_m$ which is approximately multiplicative on $\mathcal{F}$ up to $\epsilon$. 
\end{lemma}

\begin{proof}
Let $\mathcal{F}\subset\mathcal{Z}$ be finite and let $\epsilon>0$. 
We can further assume that all elements in $\mathcal{F}$ are contractions.
Since $\mathcal{Z}$ is an inductive limit of \emph{dimension drop algebras}\footnote{ $\Z_{n,n+1}=\{f\in C\left([0,1],\M_n\otimes \M_{n+1}\right): f(0)\in \M_n\otimes 1_{\M_{n+1}}, f(1)\in 1_{\M_n}\otimes \M_{n+1}\}$.} $\mathcal{Z}_{n,n+1}$ (\cite[Proposition 2.5, Theorem 2.9]{JiangSu}), there exists $n\in\mathbb{N}$ satisfying that for all $a\in \mathcal{F}$, there is a contraction $x\in\mathcal{Z}_{n,n+1}$ such that $\|a-x\|<\frac{\epsilon}{4}$. 
Denote by $\ev_0$ and $\ev_1$ the unital $^*$-homomorphisms $\ev_0:\mathcal{Z}_{n,n+1}\to \M_n$ and $\ev_1:\mathcal{Z}_{n,n+1}\to \M_{n+1}$, given by evaluation at $0$ and $1$ respectively.

By Arveson's extension theorem (see for example \cite[Corollary 1.5.16]{BO08}), we obtain ucp maps $\pi_n:\mathcal{Z}\to \M_n$ and $\pi_{n+1}:\mathcal{Z}\to \M_{n+1}$ extending $\ev_0$ and $\ev_1$, respectively. Then, for any $a,b\in\mathcal{F}$, let $x,y\in\mathcal{Z}_{n,n+1}$ be contractions such that $\|x-a\|<\frac{\epsilon}{4}$ and $\|y-b\|<\frac{\epsilon}{4}$. Then, since $\pi_n$ restricted to $\mathcal{Z}_{n,n+1}$ is a $^*$-homomorphism, a direct estimation gives that
\begin{align}
	\pi_n(a) \pi_n(b) \approx_\frac{\epsilon}{2} \pi_n(x) \pi_n(y) = \pi_n(xy) \approx_\frac{\epsilon}{2} \pi_n(ab). 
\end{align}
Hence, $\pi_{n}$ is approximately multiplicative on $\mathcal{F}$ up to $\epsilon$.
A similar argument shows
that $\pi_{n+1}$ is also approximately multiplicative on $\mathcal{F}$ up to $\epsilon$.

Take $N=n(n+2)$. As noted in Remark \ref{remark:explicit.weak.div}, for all $m\geq N$, there exist $r,s\in \mathbb{N}$ such that $m=rn+s(n+1)$. Therefore, by considering $\varphi_m: \Z \to \M_m$ as 
$\varphi_m(a) = \pi_{n}(a)^{\oplus r} \oplus \pi_{n+1}(a)^{\oplus s}$, 
we obtain a ucp map which is approximately multiplicative on $\mathcal{F}$ up to $\epsilon$.
\end{proof}

This technical lemma now allows us to completely characterise when unital embeddings of $\mathcal{Z}$ into simple, separable, unital, nuclear, $\Z$-stable $\C$-algebras have nuclear dimension equal to zero.

\begin{theorem}\label{zmorphisms}
Let $B$ be a simple, separable, unital, nuclear, $\Z$-stable $\C$-algebra and $\theta:\Z\to B$ a unital $^*$-homomorphism. Then $\dimnuc\theta=0$ if and only if $B$ is weakly divisible of degree n at $1_B$ for all $n\in\mathbb{N}$. In particular, if $B$ is not weakly divisible of degree n at $1_B$ for all $n\in\mathbb{N}$, then any unital embedding of $\Z$ into $B$ has nuclear dimension equal to one.
\end{theorem}

\begin{proof}
Suppose that $\dimnuc\theta=0$. By Corollary \ref{codomainweaklydiv}
we obtain that $B$ is weakly divisible of degree $n$ at $1$ for all $n\in\mathbb{N}$.

Let us prove the converse. Let $(\mathcal{F}_k)_{k\in\N}$ be an increasing family of finite sets of $\Z$ with dense union. 
For each $k \in \mathbb{N}$, Lemma \ref{QDofZ} yields a natural number $N_k$ such that for all $m \geq N_k$ there is a ucp map $\pi^{(k)}_{m}: \Z \to \mathbb M_{m}$ which is approximately multiplicative on $\mathcal{F}_k$ up to $\frac{1}{k}$. By assumption, $B$ is weakly divisible of degree $N_k$ at $1_B$ for any $k\in\mathbb{N}$. Then, there exist $m_1,\ldots,m_r \geq N_k$ and a unital $^*$-homomorphism $\eta_{k}:\M_{m_1}\oplus\ldots\oplus \M_{m_r}\to B$. 
To simplify the notation, set $F_{k} := \bigoplus\limits_{i=1}^r \M_{m_i}$. Then consider $\psi_k:= \pi^{(k)}_{m_1}\oplus \ldots \oplus \pi^{(k)}_{m_r}:\mathcal{Z}\to F_{k}$.
Denote by $\psi:\mathcal{Z}\to \prod\limits_{\infty} F_{k}$ and $\eta: \prod\limits_{\infty} F_{k} \to B_\infty$
the maps induced by $(\psi_k)_{k\in\mathbb{N}}$ and $(\eta_{k})_{k\in\mathbb{N}}$, respectively.
Observe that both maps are unital
 $^*$-homomorphisms. 

Let us consider now the unital $^*$-homomorphisms $\iota_B\circ\theta$ and $\eta\circ\psi$. 
We claim that these two maps are unitarily equivalent. Since $\mathcal{Z}$ has a unique trace, $\Aff(T(\iota_B\circ\theta))=\Aff(T(\eta\circ\psi))$. 
Both maps induce the same map on $K_0$ since they are unital and $K_0(\mathcal{Z})=\mathbb{Z}$. 
Similarly, they are equal to zero on $K_1$ since $K_1(\Z)=0$.

At the level of total K-theory, note that since $\mathcal{Z}$ is $\KK$-equivalent to $\mathbb{C}$, with the equivalence being induced by the canonical unital inclusion (\cite[Lemma 2.11]{JiangSu}), any unital embedding of $\mathcal{Z}$ induces the same class in $\KK(\mathcal{Z},B_\infty$). Since KL$(\Z,B_\infty)$ is a quotient of $\KK(\mathcal{Z},B_\infty$), it follows that any unital embedding of $\Z$ induces the same class in KL$(\Z,B_\infty)$. Then, by the universal multicoefficient theorem (\cite[Theorem 8.5]{classif}), we conclude that $\iota_B\circ\theta$ and $\eta\circ\psi$ induce the same $\Lambda$-morphism at the level of total K-theory. 

Finally, since $K_1(\mathcal{Z})=0$, the decomposition formula in \eqref{algK1decomp} shows that the Hausdorffised algebraic $K_1$-maps are completely determined by  
the trace maps. Therefore, $\overline{K_1}^{\alg}(\iota_B\circ\theta)= \overline{K_1}^{\alg}(\eta\circ\psi)$. Hence, $\iota_B\circ\theta$ and $\eta\circ\psi$ are unitarily equivalent by {\cite[Theorem 9.1]{classif}}.

By construction, $\dimnuc (\eta\circ\psi)=0$. Thus $\iota_B\circ\theta$ has nuclear dimension equal to zero as well since nuclear dimension is preserved by unitary equivalence. Therefore, $\dimnuc \theta=0$ by \cite[Proposition 2.5]{decomprankZstable}.
The last part of the statement holds from the fact that nuclear dimension of any embedding is bounded by the nuclear dimension of $\Z$ (see Lemma \ref{bounddimmorphism}).
\end{proof}

We will focus now on the nuclear dimension of unital embeddings of $\Z$ into $\Z$-stable $\C$-algebras of real rank zero.
Before proceeding, we recall the following classification theorem for unital maps with domain a strongly self-absorbing $\C$-algebra $\mathcal{D}$ that appears in \cite{TW07}.

\begin{prop}[{\cite[Corollary 1.12]{TW07}}]\label{prop:ssa.class.maps}
	Let $\mathcal{D}$ be a strongly self-absorbing $\C$-algebra and let $A$ be a $\mathcal{D}$-stable unital $\C$-algebra. Then all unital $^*$-homomorphisms from $\mathcal{D}$ to $A$ are mutually approximately unitarily equivalent.
\end{prop}

With the previous statement in hand, we can compute the nuclear dimension of unital embeddings of $\Z$ into $\Z$-stable $\C$-algebras of real rank zero.

\begin{prop}\label{prop:Zstable.0.dim.maps}
	Let $A$ be a unital $\Z$-stable $\C$-algebra of real rank zero. Then any unital $^*$-homomorphism from $\Z$ to $A$ is approximately unitarily equivalent to a map which factors through a simple unital AF-algebra. 
	In particular, any unital $^*$-homomorphism from $\Z$ to $A$ has nuclear dimension equal to zero. 
\end{prop}

\begin{proof}
	By \cite[Corollary 12]{embZ}, there is a simple separable unital infinite dimensional AF-algebra $C$ and a unital embedding $\varphi:C\to A$. Since $C$ is $\Z$-stable, there is a unital $^*$-homomorphism $\psi:\Z \to C$. Hence, $\varphi \circ \psi: \Z \to A$ is a unital $^*$-homomorphism with nuclear dimension equal to zero. By Proposition \ref{prop:ssa.class.maps}, any other unital $^*$-homomorphism $\Z \to A$ is approximately unitarily equivalent to $\varphi \circ \psi$; and hence, it also has nuclear dimension equal to zero. 
\end{proof}

We finish this paper by showing that if the trace space of $B$ has at most two extreme points, any unital embedding of $\mathcal{Z}$ into $B$ with nuclear dimension equal to zero is approximately unitarily equivalent to a unital $^*$-homomorphism which factors through a simple AF-algebra. The proof boils down to constructing a dimension group within $K_0(B)$ using that the unit is weakly divisible. However, even when $B$ has three extreme traces, it is not apparent how to build such a dimension group.

\begin{theorem}
Let $\theta:\mathcal{Z}\to B$ be a unital $^*$-homomorphism into a simple, separable, unital, nuclear, $\mathcal{Z}$-stable $\C$-algebra $B$ which has at most two extreme traces. Then the following are equivalent:

\begin{enumerate}
\item $\theta$ is approximately unitarily equivalent to a $^*$-homomorphism $\Z\to B$ which factors through a simple separable unital AF-algebra;
\item $\theta$ has nuclear dimension equal to zero;
\item for all $n\in\mathbb{N}$ there exist $g^{(n)},h^{(n)}\in K_0(B)_+$ such that 
\begin{align} \label{theo:2traces.eq1}
[1_B]_0=ng^{(n)}+(n+1)h^{(n)}.
\end{align}
\end{enumerate}
\end{theorem}

\begin{proof}
Combining Remark \ref{remark:explicit.weak.div} and Theorem \ref{zmorphisms}, it follows that (ii) and (iii) are equivalent. Since (i) implies (ii), it suffices to check that (iii) implies (i).

If $B$ has no bounded traces, then $B$ is purely infinite ({\cite[Corollary 5.1]{rr0char}}). Since $B$ is also simple, it has real rank zero by \cite{Zha90}. Then, by Proposition \ref{prop:Zstable.0.dim.maps}, $\theta$ is approximately unitarily equivalent to a map which factors through a simple unital AF-algebra.

Suppose now that $B$ has a unique trace. Then the weak divisibility at the unit implies that $\rho_B(K_0(B))$ is a subgroup of $\mathbb{R}$ with arbitrarily small positive elements. This means that $\rho_B(K_0(B))$ is uniformly dense in $\Aff(T(B))$  and hence $B$ has real rank zero by {\cite[Theorem 7.2]{rr0char}}. As before, Proposition \ref{prop:Zstable.0.dim.maps} yields that $\theta$ is approximately unitarily equivalent to a map which factors through a simple unital AF-algebra.

Suppose that $B$ has two extreme traces $\tau_1$ and $\tau_2$. For simplicity, let us write $\tau_i$ instead of $K_0(\tau_i)$.
Let us consider first the case when
$\tau_1(g^{(n)})=\tau_2(g^{(n)})$ for all but finitely many $n\geq 2$. 
Then, by equation \eqref{theo:2traces.eq1}, the corresponding $h^{(n)}$'s must be constant in traces as well (i.e.\ $\tau_1(h^{(n)})=\tau_2(h^{(n)})$). 
We will  consider the subgroup $H \subset K_0(B)$ generated by those  $g^{(n)}$ and $h^{(n)}$ which are constant in traces. Let us equip $H$ with the order induced by $\tau_1$ and note that $H$ has arbitrarily small positive elements in trace. Take $C$ to be some classifiable $\C$-algebra with unique trace, $(K_0(C), K_0(C)_+, [1_C]_0)\cong (H,H_+, [1_B]_0)$ and $K_1(C)=0$. 
Again, since $\rho_C(K_0(C))$ is a subgroup of $\mathbb{R}$ with arbitrarily small positive elements, $C$ has real rank zero by {\cite[Theorem 7.2]{rr0char}}.
Similarly, $\theta$ is approximately unitarily equivalent to a $^*$-homomorphism factoring through $C$.
As before,by Proposition \ref{prop:Zstable.0.dim.maps} any unital embedding of $\mathcal{Z}$ into $C$ is approximately unitarily equivalent to a map which factors through a simple separable unital AF-algebra. Hence, the conclusion follows. 

We are left with the case when there are infinitely many $n\in\mathbb{N}$ such that $\tau_1(g^{(n)})\neq \tau_2(g^{(n)})$. By passing to a subsequence, we can assume that $\tau_1(g^{(n)})\neq \tau_2(g^{(n)})$ for all $n\in\mathbb{N}$. 
Observe that this immediately implies $\tau_1(h^{(n)})\neq \tau_2(h^{(n)})$  for all $n\in\mathbb{N}$. 

We claim that $\rho_B(K_0(B))$ is dense in $\mathbb{R}^2$ and hence $B$ has real rank zero by {\cite[Theorem 7.2]{rr0char}}. Let $f:\mathbb{R}^2\to\mathbb{R}$ be any non-zero linear functional. We will show that $f(\rho_B(K_0(B)))$ is dense in $\mathbb{R}$. 
Consider $\epsilon >0$ and say $f((1,0))=\alpha_1$ and $f((0,1))=\alpha_2$. Then $f(\rho_B(x))=\alpha_1\tau_1(x)+\alpha_2\tau_2(x)$ for any $x\in K_0(B)$. 
On the other hand, by equation \eqref{theo:2traces.eq1} we have
\begin{align}
	\alpha_1 + \alpha_2 = f(\rho_B ([1_B]_0)) = n f(\rho_B (g^{(n)})) + (n+1)f(\rho_B (h^{(n)})). \notag
\end{align}
If $\alpha_1 + \alpha_2 \neq 0$, then either $f(\rho_B (g^{(n)})) \neq 0$ or $f(\rho_B (h^{(n)})) \neq 0$. In both cases, if we consider the element with non-zero image and choose a large enough $n$, we obtain an element $x \in K_0(B)$ such that
\begin{align}
	0 < |f(\rho_B (x))| < \epsilon.
\end{align} 
On the other hand, if $\alpha_1+\alpha_2=0$, then $$|f(\rho_B(g^{(n)}))|=|\alpha_1(\tau_1(g^{(n)})-\tau_2(g^{(n)}))|<\epsilon$$ for $n$ large enough and it is strictly positive since $\tau_1(g^{(n)})\neq \tau_2(g^{(n)})$. 

By considering $-g^{(n)}$ or $-h^{(n)}$ if needed, we conclude that the image of $\rho_B(K_0(B))$ through $f$ is a subgroup of $\mathbb{R}$ which contains arbitrarily small positive elements and hence it is dense.
Then, it follows that $\rho_B(K_0(B))$ is dense in $\mathbb{R}^2$  by {\cite[Theorem 1.1]{densesubgroupsRn}}. Thus, $B$ has real rank zero and the result follows by Proposition \ref{prop:Zstable.0.dim.maps}.
\end{proof}

\end{document}